 \RequirePackage{fix-cm}

 \documentclass[smallextended]{svjour3} 
 \smartqed  

\usepackage{amssymb,amsmath,amsfonts,latexsym} 
\usepackage{empheq}
\usepackage{hyperref}
\usepackage{tikz}
\usepackage{lipsum}
\usepackage{color, xcolor}
\usepackage{cases}
\usepackage{cite}

\counterwithin*{figure}{section}

\usepackage{multirow,array}
\usepackage{graphicx}
\usepackage{subcaption}
\usepackage{float}
\usetikzlibrary{arrows.meta}
\renewcommand\thesubfigure{\thefigure-\roman{subfigure}}
\numberwithin{equation}{section}
\allowdisplaybreaks

\makeatletter
\renewcommand{\makeheadbox}{%
\hbox to0pt{\vbox to0pt{\vss}\hss}%
}
\makeatother
\begin{document}


\title{On the asymptotic behaviour of the Fourier transform of the Mittag-Leffler function}

\titlerunning{On the asymptotic behaviour of the Fourier transform}

\author{Ahmed A. Abdelhakim}

\authorrunning{A. A. Abdelhakim} 

\institute{Ahmed A. Abdelhakim
\at
Department of Mathematics, College of Science, Qassim University, Buraydah, 51452, Saudi Arabia
\\
Department of Mathematics, Faculty of Science, Assiut University,  Assiut 71516, Egypt \\
\email{ahmed.abdelhakim@aun.edu.eg}  
}

\date{}


\maketitle

\begin{abstract}
Let $\alpha \in (0,2)$ and let $\beta>0$.
Fix $-\pi<\varphi\leq \pi$ such that
$|\varphi|>\alpha \pi/2$. We obtain asymptotic upper bounds  on the Fourier transform of the radially symmetric tempered distribution
\begin{equation*}
\mathbb{R}^n\ni x\mapsto E_{\alpha,\beta}(e^{\dot{\imath} \varphi} |x|^{\sigma}),
\end{equation*}
for $\sigma>(n-1)/2$, where $E_{\alpha,\beta}$ is the two-parameter Mittag-Leffler function. As an application, we obtain some values of the Lebesgue exponent $p=p(\sigma)$, $\sigma>(n-1)/2$, for which the Fourier transform is in $L^{p}(\mathbb{R}^{n})$. Such values cannot be obtained via the well-known $L^{p}(\mathbb{R}^{n})$ properties of $E_{\alpha,\beta}$ and the Hausdorff-Young inequality, when $\sigma\leq n/2$.


\end{abstract} 


\section{Introduction and main results}
\indent Let $\alpha \in (0,2)$, $\beta>0$, and let $\mathcal{F}$ denote the Fourier transform operator. Solutions to various space-time fractional problems are, roughly speaking, a convolution operator with the radially symmetric function
\begin{equation*}
\mathbb{R}^n\times(0,\infty)\ni (\xi,t)\mapsto t^{a}\mathcal{F} E_{\alpha,\beta}(e^{\dot{\imath} \varphi} t^{b}|\cdot|^{\sigma})(\xi),
\end{equation*}
for some $a,b\in \mathbb{R}$ and $\sigma>0$, in some fixed direction $-\pi<\varphi\leq \pi$ (see e.g.
\cite{chen2012space,dong2008space,
grande2019space,huang2005space,kemppainen2017representation,
lee2020strichartz,li2022applicationi,
li2022applicationii,mainardi2001fundamental,mainardi2003wright,
wang2007generalized}). Determining the decay rate of the spatial  $L^{p}$ norm of the solution is often a crucial step towards obtaining space-time estimates (see e.g.
\cite{lee2020strichartz,
li2022applicationi,li2022applicationii,su2021dispersive}). Such estimates are useful
for understanding long-time behaviour of the solution
(see e.g.\cite{kemppainen2017representation}).
They are also indispensable for studying well-posedness of the corresponding semilinear equation ( see e.g. \cite{grande2019space,su2019local}). This places great importance on understanding the integrability properties of the Fourier transform of $E_{\alpha,\beta}(e^{\dot{\imath} \varphi} |\cdot|^{\sigma})$. \par\medskip
It is well known that $E_{\alpha,\beta}$
is an entire function of order $1/\alpha$, when $\alpha,\beta>0$ (see e.g. \cite{gorenflo2020mittag}, Section 4.1). The function $x\mapsto E_{\alpha,\beta}(e^{\dot{\imath} \varphi } |x|^{\sigma})$
is therefore continuous on $\mathbb{R}^{n}$, whenever $\sigma >0$. It is in fact smooth away from the origin. The values of $1\leq p\leq \infty$ for which it is in $L^{p}(\mathbb{R}^{n})$ are thus determined by its asymptotic behaviour as $|x|\rightarrow \infty$. We shall restrict our attention to the sector $|\varphi|>\pi \alpha/2$.
Notice that $E_{\alpha,\beta}\left(e^{\dot{\imath} \varphi } |\cdot|^{\sigma}\right)\notin L^{p}(\mathbb{R}^{n})$, for any $1\leq p\leq \infty$, when $|\varphi|<\pi \alpha/2$. In that sector (see e.g. \cite{kilbas2006theory}, Theorem 4.3, and \cite{podlubny1998fractional}, Theorem 1.3),
\begin{equation*}
|E_{\alpha,\beta}\left(e^{\dot{\imath} \varphi}
|x|^{\alpha\sigma}\right)|\sim |x|^{\sigma(1-\beta)}
e^{|x|^{\sigma} \cos{\left({\varphi}/{\alpha}\right)}},\quad |x|\rightarrow \infty.
\end{equation*}
If $|\varphi|=\pi \alpha/2$, $E_{\alpha}\left(e^{\dot{\imath} \varphi}
|\cdot|^{\sigma}\right)$ exhibits an oscillatory behaviour
in the sense that it changes sign while its magnitude does not decay fast enough for it to be in some $L^{p}(\mathbb{R}^{n})$ space (see e.g. \cite{abdelhakim2023asymptotic,garrappa2017time}). On the other hand, in the sector $|\varphi|>\pi \alpha/2$, we have the estimate (see e.g. \cite{podlubny1998fractional}, Theorem 1.6):
\begin{equation*}
\left|E_{\alpha,\beta}\left(e^{\dot{\imath} \varphi } |x|^{\sigma}\right)\right|
\lesssim |x|^{-\sigma},\quad |x|>1,
\end{equation*}
for all $\alpha\in (0,2)$, and all $\beta>0$. Thus, $E_{\alpha,\beta}\left(e^{\dot{\imath} \varphi } |\cdot|^{\sigma}\right)\in L^{p}(\mathbb{R}^{n})$, for all ${n}/{\sigma}<p\leq \infty$. Theorem \ref{thmmph1} below follows from this fact and the Hausdorff-Young inequality.
\begin{theorem}\label{thmmph1}
Let $\alpha\in (0,2)$ and $\beta>0$.
Suppose that $|\varphi|>\pi \alpha/2$. Then
the Fourier transform of $E_{\alpha,\beta}(e^{\dot{\imath} \varphi } |\cdot|^{\sigma})$ is in $L^{p}(\mathbb{R}^{n})$, for all
\begin{flalign*}
&2\leq p\leq \infty,\qquad  \text{if}\quad \sigma>n,\\
&2\leq p< \infty, \qquad \text{if}\quad\sigma=n, \\
&2\leq p< {n}/{(n-\sigma)},\qquad
\text{if}\quad {n}/{2}<\sigma<n.
  \end{flalign*}
\end{theorem}

The Hausdorff-Young inequality is not useful, however, when $\sigma\leq n/2$. In that case, $E_{\alpha,\beta}\left(e^{\dot{\imath} \varphi } |\cdot|^{\sigma}\right) \notin L^{p}(\mathbb{R}^{n})$, for  any $1\leq p\leq 2$. Nevertheless, it is a continuous bounded function for all $\sigma>0$ and has, therefore, a Fourier transform in the sense of tempered distributions. This begs the questions:
\begin{enumerate}
  \item  For what values of $p=p(\sigma)$, $p(\sigma)\in [1,2)$ is the Fourier transform of
$E_{\alpha,\beta}\left(e^{\dot{\imath} \varphi } |\cdot|^{\sigma}\right)$ in $L^{p}(\mathbb{R}^{n})$ ?
  \item Given $\sigma\leq n/2$, for which values of $p$, if any, is the Fourier transform of
$E_{\alpha,\beta}\left(e^{\dot{\imath} \varphi } |\cdot|^{\sigma}\right)$ in $L^{p}(\mathbb{R}^{n})$  ?
\end{enumerate}
Motivated by these questions, we investigate the asymptotic behaviour of
$\mathcal{F} E_{\alpha,\beta}(e^{\dot{\imath} \varphi} |\cdot|^{\sigma})(\xi)$ both as $\xi\rightarrow 0$
and as $|\xi|\rightarrow \infty$, for $\alpha\in (0,2)$ and $\beta>0$, when $|\varphi|>\pi \alpha/2$. This is done under the assumption that $\sigma>(n-1)/2$. A byproduct of our approach is the observation that $\mathcal{F} E_{\alpha,\beta}(e^{\dot{\imath} \varphi} |\cdot|^{\sigma})$ is continuous almost everywhere on $\mathbb{R}^{n}$, in the aforedescribed range of values of the parameters.
\begin{theorem}\label{thmsmt}
Assume that $\alpha\in (0,2)$ and $\beta>0$.
Fix $\,-\pi<\varphi\leq \pi$ such that
$|\varphi|>\alpha \pi/2$. Then
\begin{equation}\label{eab1}
{\mathcal{F}
\left(E_{\alpha,\beta}(e^{\dot{\imath}
\varphi } |\cdot|^{\sigma})
\right)(\xi)=}
\begin{cases}
O(|\xi|^{\sigma-n}),& (n-1)/2<\sigma<n,\medskip\\
O(\log{|\xi|}),& \sigma=n,\medskip\\
O(1),& \sigma>n,
\end{cases}
\end{equation}
as $\xi \rightarrow 0$. Moreover, for any $\sigma>(n-1)/2$,
\begin{equation}\label{eab2}
\mathcal{F}\left(E_{\alpha,\beta}(e^{\dot{\imath} \varphi } |\cdot|^{\sigma})
\right)(\xi)=O(|\xi|^{-n}),\quad |\xi| \rightarrow \infty.
\end{equation}
\end{theorem}

Together, the asymptotic formulas (\ref{eab1})
and (\ref{eab2}) imply the estimates
\begin{equation*}
{\left|\mathcal{F}\left(E_{\alpha,\beta}(e^{\dot{\imath} \varphi } |\cdot|^{\sigma})
\right) (\xi)\right|\lesssim}
\begin{cases}
|\xi|^{\sigma-n}(1+|\xi|^{\sigma})^{-1},&
(n-1)/2<\sigma<n,\medskip\\
\left(1+{|\xi|^{n}}\log{|\xi|}\right)^{-1}
\log{|\xi|},& \sigma=n,\medskip\\
\left(1+{|\xi|^{n}}\right)^{-1},& \sigma>n,
\end{cases}
\end{equation*}
and we arrive at the following $L^{p}(\mathbb{R}^{n})$ boundedness result for the Fourier transform of $E_{\alpha,\beta}(e^{\dot{\imath} \varphi } |\cdot|^{\sigma})$:
\begin{theorem}
Suppose that
$\alpha$, $\beta$ and
$\varphi$ are as in Theorem \ref{thmsmt}.
Then the Fourier transform of
 $E_{\alpha,\beta}(e^{\dot{\imath}
\varphi } |\cdot|^{\sigma})$ is in $L^{p}(\mathbb{R}^{n})$, for all
\begin{equation*}
 \left\{
   \begin{array}{ll}
 1<p<{n}/{(n-\sigma)}, & \hbox{$(n-1)/2<\sigma<n$;}
\medskip \\
    1<p<\infty, & \hbox{$\sigma=n$;} \medskip\\
        1<p\leq\infty, & \hbox{$\sigma>n$.}
   \end{array}
 \right.
\end{equation*}
\end{theorem}

The rest of this paper is dedicated to the proof of
Theorem \ref{thmsmt}. Our starting point is the identity
\begin{equation*}
\mathcal{F}\left(E_{\alpha,\beta}(e^{\dot{\imath} \varphi } |\cdot|^{\sigma})
\right)(\xi)
=(2\pi)^{\frac{n}{2}}|\xi|^{1-\frac{n}{2}}
\int_{0}^{\infty}
E_{\alpha,\beta}(e^{\dot{\imath} \varphi } r^{\sigma})
J_{\frac{n}{2}-1}( r |\xi|)
r^{\frac{n}{2}} dr,
\end{equation*}
where $J_{{n}/{2}-1}$ denotes Bessel function of the first kind.
This identity follows from applying the well-known formula
(\ref{fjdl1})
in Section \ref{sec2} to the radially symmetric tempered distribution
$E_{\alpha,\beta}(e^{\dot{\imath} \varphi} |\cdot|^{\sigma})$
defined on $\mathbb{R}^{n}$.
The proof of the asymptotic formulas (\ref{eab1})
and (\ref{eab2}) is given in Section \ref{prcs}.
In a nutshell, they follow from the identity
(\ref{frtj}), Lemma \ref{lmsmpt1} and Lemma \ref{lmsmpt2} combined. The idea of the proof of these lemmas is discussed at the beginning of Section \ref{prcs}.
We  present some numerical tests of
the asymptotic formula (\ref{eab1}) of Theorem \ref{thmsmt}
in Appendix A at the end of the paper.
In the upcoming Section 2,
we collect some useful properties of the Mittag-Leffler function
and Bessel function of the first kind,
and we prove some preliminary lemmas.
\section{Preliminaries and auxiliary results}
\subsection{ \textbf{A suitable integral representative and
some useful estimates}}\label{sec1}
Fix $\alpha\in (0,2)$ and $\beta\in \mathbb{C}$
and let $\epsilon>0$. Let
$\pi \alpha/2<\omega< \pi \alpha$,
if $0<\alpha<1$, and $\pi \alpha/2<\omega< \pi$,
if $1\leq \alpha<2$, i.e., let
\begin{equation*}
\pi \alpha/2<\omega< \min\{\pi \alpha,\pi\}.
\end{equation*}
Consider the path $H_{\epsilon, \omega}$ in the complex plane that consists of the two rays
$\left\{\arg{z}=-\omega, |z|\geq \epsilon\right\}$ and
$\left\{\arg{z}=\omega, |z|\geq \epsilon\right\}$,
and the circular arc
\begin{equation*}
\left\{-\omega\leq\arg{z}\leq \omega, |z|=\epsilon\right\},
\end{equation*}
joined and taken in the direction demonstrated in Figure \ref{figcntr0} below (cf. \cite{gorenflo2002computation}, Figure 1 and  \cite{podlubny1998fractional}, Figure 1.4).
The Mittag-Leffler function $E_{\alpha,\beta}$ has the following contour integral representation:
\begin{equation}\label{cntr0}
E_{\alpha,\beta}(z)=
\frac{1}{2\pi \dot{\imath} \alpha}\int_{H_{\epsilon,\omega}}
\frac{e^{\varrho^{1/\alpha}}{
\varrho^{(1-\beta)/\alpha}}}{
\varrho-z}d\varrho,
\end{equation}
for all $z\in \mathbb{C}$ such that
$|z|<\epsilon$ or $|\arg{z}|>\omega$.
\begin{figure}[H]
\centering
\begin{tikzpicture}[scale=0.45]
\fill [gray!10] (0,0) circle (2.5);
\fill [gray!10] (0.434,2.462)--(1.22,6.89)--
(-7,7)--(-7,-7)--(1.22,-6.89)
--(0.434,-2.462)--cycle;
\draw [->](-7,0) --(7,0);
\draw [->](0,-7) --(0,7);
\draw [dashed](0,0)--(0.434,2.462);
\draw [thick](0.434,2.462)--(1.22,6.89);
\draw [-{Latex[length=2mm]}](0.543,3.08) --(0.76,4.31);
\draw [dashed](0,0)--(0.434,-2.462);
\draw [thick](0.434,-2.462)--(1.22,-6.89);
\draw [-{Latex[length=2mm]}](0.76,-4.31)--(0.543,-3.08);
\draw [thick] (2.5,0) arc [radius=2.5, start angle=0, end angle= 80];
\draw [thick] (2.5,0) arc [radius=2.5, start angle=0, end angle= -80];
\draw  [-{Latex[length=2mm]}] (2.35,-0.853) arc [radius=2.5, start angle=-27, end angle= -25];
\draw (0.45,0) arc [radius=0.55, start angle=0, end angle= 70];
\node [above right] at (2.28,-0.18) {\small $\epsilon$};
\node [above] at (0.75,-0.161) {\small $\omega$};
\node [right] at (1.9,-1.7678) {$H_{\epsilon,\omega}$};
\end{tikzpicture}
\caption{\small The contour $H_{\epsilon, \omega}$ ($\alpha \in (0,1/2)$). The integral (\ref{cntr0}) represents $E_{\alpha,\beta}$ in the shaded region.}\label{figcntr0}
\end{figure}
The representation (\ref{cntr0}) is derived originally in \cite{dzhrbashyan1954integral}.
See also \cite{gorenflo2020mittag}, Section 4.7,
\cite{gorenflo2002computation}, and
\cite{podlubny1998fractional}, Theorem 1.1, for more details. In particular, given
$-\pi<\varphi\leq \pi$ such that
$\,|\varphi|> \pi\alpha/2$, and $r\geq 0$, one may set
\begin{equation*}
\pi\alpha/2<\omega<\min
\left\{|\varphi|,\pi \alpha \right\},
\end{equation*}
and use (\ref{cntr0}) to write
\begin{equation}\label{esln1}
E_{\alpha,\beta}(r e^{\dot{\imath} \varphi })=
\frac{1}{2\pi \dot{\imath} \alpha}\int_{H_{\omega}}
\frac{e^{z^{1/\alpha}}{
z^{(1-\beta)/\alpha}}}{
z-r e^{\dot{\imath} \varphi }}dz,
\end{equation}
for all $\alpha \in (0,2)$ and $\beta\in \mathbb{R}$, where
$H_{\omega}$ denotes for short the contour $H_{1,\omega}$.
The contour integral representative (\ref{esln1}) is used in \cite{abdelhakim2025mittag} to show that $E_{\alpha,\beta}( e^{\dot{\imath} \varphi } |x|^{\sigma})$ is an eigenfunction of the Hilbert transform on $L^{p}(\mathbb{R})$, for any $\sigma>0$,
provided that $|\varphi|>\alpha\pi/2$.
\par\medskip
It is well-known (see e.g. \cite{kilbas2006theory}, Section 1.8, \cite{podlubny1998fractional}, Theorem 1.6, and \cite{abdelhakim2023asymptotic}) that there exists a positive constant $C$ such that
\begin{equation}\label{ptws0}
|E_{\alpha,\beta}(z)|\leq
C(1+|z|)^{-1},\quad
\alpha\pi/2< |\arg{z}|\leq \pi,
\end{equation}
for all $\alpha \in (0,2)$ and all $\beta\in \mathbb{R}$.\bigskip\par\medskip
The following elementary lemma will be useful:
\begin{lemma}\label{infrt}
Assume that
\begin{equation*}
0<\alpha<2,\;\;-\pi<\varphi\leq \pi,\;\;|\varphi|> \pi\alpha/2,\;\;\;
\pi\alpha/2<\omega<\min
\left\{|\varphi|,\pi \alpha \right\}.
\end{equation*}
Let $H_{\omega}$ be the contour that consists of the two rays
\begin{equation*}
\left\{\arg{z}=-\omega, |z|\geq 1\right\}\quad \text{and}\quad
\left\{\arg{z}=\omega, |z|\geq 1\right\},
\end{equation*}
and the circular arc $\left\{-\omega\leq\arg{z}\leq \omega, |z|=1\right\}$, traced in the direction shown in
Figure \ref{figcntr0} above with $\epsilon=1$.
Let $\varphi_{0}$ be the constant given by
\begin{equation*}
8\varphi_{0}=\begin{cases}
\min{\{|\sin{(\varphi-\omega)}|,|\sin{(\varphi+\omega)}|\}},&
|\varphi- \omega|\neq \pi,\;|\varphi+ \omega|\neq \pi;\vspace{0.15 cm}\\
|\sin{(\varphi-\omega)}|,&
|\varphi- \omega|\neq \pi,\;|\varphi+ \omega|= \pi;\vspace{0.15 cm}\\
|\sin{(\varphi+\omega)}|,&
|\varphi- \omega|= \pi,\;|\varphi+ \omega|\neq \pi;\vspace{0.15 cm}\\
1,&
|\varphi- \omega|=\pi,\;|\varphi+ \omega|=\pi.
\end{cases}
\end{equation*}
Then, we have
\begin{equation}\label{zrst1}
\min_{z\in H_{\omega}}{|z-r e^{\dot{\imath} \varphi }|}
\geq \varphi_{0}\max{\{r,4\}}.
\end{equation}
In particular, we have
\begin{eqnarray}
\label{zrst11}\min_{\rho\geq 1}{|\rho^{\alpha}-r e^{\dot{\imath} (\varphi\pm \omega )}|}&\geq& \varphi_{0}\max{\{r,4\}},\\
\label{zrst12}\min_{|\theta|\leq \omega}{|1-r e^{\dot{\imath} (\varphi-\theta )}|}&\geq& \varphi_{0}\max{\{r,4\}}.
\end{eqnarray}
\end{lemma}
\begin{proof}
Fix $r\geq 0$. First let $\alpha\pi/2<\varphi\leq \pi$. Then
\begin{equation*}
0<\varphi-\omega<(1-\alpha/2)\pi,\quad
\alpha\pi<\varphi+\omega<2\pi.
\end{equation*}
Since $\rho\mapsto r^2+\rho^2-2 r \rho \cos{(\varphi \mp \omega)}$ is strictly increasing on $[1,\infty)$ when $ r\cos{(\varphi\mp \omega)}\leq 1$, and has a unique global minimum at
$\rho=r\cos{(\varphi\mp \omega)}$ if $r\cos{(\varphi\mp \omega)}>1$, we have
\begin{multline}\label{inf1}
\min{\{|r e^{\dot{\imath} \varphi }-\rho e^{\pm\dot{\imath} \omega }|^{2}:\rho\geq 1\}}\\
\qquad\quad\;=\min{\{r^2+\rho^2-2 r \rho \cos{(\varphi \mp \omega)}:\rho\geq 1\}}\vspace{0.2 cm}\\
=
\begin{cases}
1+r^{2}-2 r \cos{(\varphi\mp \omega)},&r\cos{(\varphi\mp \omega)}\leq 1;\vspace{0.15 cm}\\
r^{2}\sin^{2}{(\varphi\mp \omega)},&r\cos{(\varphi\mp \omega)}>1.
\end{cases}
\end{multline}
Notice also that, for any real number $t$,
\begin{equation}\label{ntt}
1+r^2-2r\cos{t}\geq
\begin{cases}
1+r^{2},&\cos{t}\leq 0;\vspace{0.15 cm}\\
1+r^{2}/2,&\cos{t}>0,\;r\geq 4\cos{t};\vspace{0.15 cm}\\
\sin^{2}{t},&\cos{t}>0,\;r<4\cos{t}.
\end{cases}
\end{equation}
Using this inequality, it follows from (\ref{inf1}) that
\begin{multline}\label{inf12}
\min{\{|r e^{\dot{\imath} \varphi }-\rho e^{\pm\dot{\imath} \omega }|^{2}:\rho\geq 1\}}\vspace{0.2 cm}\\
\geq
\begin{cases}
1+r^{2},&\cos{(\varphi\mp \omega)}\leq 0;\vspace{0.15 cm}\\
1+r^{2}/2,&\cos{(\varphi\mp \omega)}>0,\;r\geq 4;\vspace{0.15 cm}\\
\sin^{2}{(\varphi\mp \omega)},&\cos{(\varphi\mp \omega)}> 0,\;r\leq \min{\{4,\sec{(\varphi\mp \omega)}\}};\vspace{0.15 cm}\\
r^{2}\sin^{2}{(\varphi\mp \omega)},&\cos{(\varphi\mp \omega)}> 0,\;\sec{(\varphi\mp \omega)}< r<4.
\end{cases}
\end{multline}
Moreover, since $\theta\mapsto 1+r^2-2 r \cos{(\varphi - \theta)}$ is strictly decreasing on $[-\omega,\omega]$ if $\varphi+\omega\leq \pi$, and has one stationary point (a maximum) at $\theta=\varphi-\pi$ if $\varphi+\omega>\pi$, we have
\begin{multline}\label{inf2}
\min{\{|r e^{\dot{\imath} \varphi }- e^{\dot{\imath} \theta }|^{2}: \theta \leq |\omega|\}}
=\min{\{1+r^2-2 r \cos{(\varphi - \theta)}:\theta \leq |\omega|\}}\\
=
\begin{cases}
1+r^{2}-2 r \cos{(\varphi-\omega)},&\varphi+\omega\leq \pi;\vspace{0.15 cm}\\
1+r^{2}-2 r \max{\{\cos{(\varphi- \omega)},\cos{(\varphi+ \omega)}\}},&\varphi+\omega>\pi.
\end{cases}
\end{multline}
Again, employing (\ref{ntt}), we deduce from (\ref{inf2}) that
\begin{multline}\label{inf22}
\min{\{|r e^{\dot{\imath} \varphi }- e^{\dot{\imath} \theta }|^{2}: \theta \leq |\omega|\}}\\
\geq
\begin{cases}
1+r^{2}/2,&r\geq 4;\vspace{0.15 cm}\\
\sin^{2}{(\varphi-\omega)},&\varphi+\omega\leq \pi,\; r< 4;\vspace{0.15 cm}\\
\min{\{\sin^{2}{(\varphi- \omega)},\sin^{2}{(\varphi+ \omega)}\}},&\varphi+\omega>\pi,\;r< 4.
\end{cases}
\end{multline}
If $-\pi<\varphi<-\alpha\pi/2$, one can obtain the minima in
(\ref{inf1}) and (\ref{inf2}), analogously, with $\varphi\mp \omega$ replaced by $\varphi\pm \omega$. The estimate (\ref{zrst1}) follows therefore from (\ref{inf12}), (\ref{inf22}), and their analogues for $\varphi\in(-\pi,-\alpha\pi/2)$.\qed
\end{proof}
\begin{remark}\!\!\!\!.
The estimate (\ref{zrst1}) and its consequences (\ref{zrst11}) and (\ref{zrst12}) will be used frequently in the proof of Theorem \ref{thmsmt} in the upcoming section.
It can be seen as a refined version of
the widely used (see e.g. \cite{podlubny1998fractional}, proof of Theorem 1.4, \cite{gorenflo2020mittag}, proof of Proposition 3.6) estimate that reads: if
\begin{equation*}
\alpha\pi/2<\omega<\mu<\min{\{\pi,\alpha\pi\}},
\end{equation*}
and $\varrho\in \mathbb{C}$ is such that
$|\varrho|$ is large and $\mu\leq |\arg{\varrho}|\leq \pi$,
then
\begin{equation*}
\min_{z\in H_{\omega}}{|\varrho-z|}
\geq |\varrho|\sin{(\mu-\omega)}.
\end{equation*}
The advantage is that (\ref{zrst1})  takes into account
the case $|\rho|$ is small.
\end{remark}
\subsection{\textbf{Bessel functions of the first kind: basics and an extended asymptotic expansion}}
\label{sec2}
The Fourier transform of a radially symmetric function $f:\mathbb{R}^{n}\rightarrow \mathbb{C}$
such that $f(x)=f_{0}(|x|)$
is the radially symmetric function given by
\begin{equation}\label{fjdl1}
\mathcal{F}{f}(\xi)=
(2\pi)^{\frac{n}{2}}|\xi|^{1-\frac{n}{2}}
\int_{0}^{\infty}
f_{0}(r)J_{\frac{n}{2}-1}( |\xi| r)
r^{\frac{n}{2}}dr,
\end{equation}
where, for $\operatorname{Re}\lambda>-{1}/{2}$, $J_{\lambda}$ is the Bessel function defined by
\begin{equation*}
J_{\lambda}(r):=
\tfrac{2^{-\lambda}}{\Gamma{(\frac{1}{2})}
\Gamma{(\lambda+\frac{1}{2})}}
r^{\lambda}
\int_{-1}^{1}e^{\dot{\imath}  r s}
(1-s^{2})^{\lambda-\frac{1}{2}}ds,\quad r\geq 0.
\end{equation*}
When $\lambda=-1/2$ we have the identity
\begin{equation}\label{jd1}
J_{-\frac{1}{2}}(r)=
\sqrt{\tfrac{2}{\pi }}{r^{-\frac{1}{2}}}{\cos{r}},\quad r> 0.
\end{equation}
If $\operatorname{Re}\lambda>-{1}/{2}$,
the asymptotic behaviour of the Bessel function $J_{\lambda}$ near zero and near infinity can be summarized as follows: \\
\indent For a small argument, we have
\begin{equation*}
J_{\lambda}(r)=\frac{1}{2^{\lambda}
\Gamma{(\lambda+1)}}r^{\lambda}+O(r^{1+\operatorname{Re}
\lambda}).
\end{equation*}
When $\operatorname{Re}\lambda>-{1}/{2}$, this comes from the identity
\begin{equation*}
J_{\lambda}(r)=\frac{1}{2^{\lambda}
\Gamma{(1+\lambda)}}r^{\lambda}+\widetilde{J}_{\lambda}(r),
\end{equation*}
where
\begin{equation*}
\widetilde{J}_{\lambda}(r):=
\frac{1}{2^{\lambda}
\Gamma{(\lambda+\frac{1}{2})}\Gamma{(\frac{1}{2})}}r^{\lambda}
\int_{-1}^{1}(e^{\dot{\imath}rt}-1)
(1-t^{2})^{\lambda-\frac{1}{2}}dt,
\end{equation*}
and $\widetilde{J}_{\lambda}$ satisfies the inequality
\begin{equation*}
|\widetilde{J}_{\lambda}(r)|\leq
\frac{1}{
2^{\operatorname{Re}\lambda}(1+\operatorname{Re}\lambda)
|\Gamma{(\lambda+\frac{1}{2})}|\Gamma{(\frac{1}{2})}}\,
r^{1+\operatorname{Re}\lambda}.
\end{equation*}
In particular, for $n>1$, we have
\begin{equation}\label{jdsmll1}
J_{\frac{n}{2}-1}(r)=
a_{n}r^{\frac{n}{2}-1}
+O(r^{\frac{n}{2}}),
\end{equation}
where $a_{n}:={2^{1-\frac{n}{2}}}/{\Gamma{(\frac{n}{2})}}$.\\
\indent For a large argument, we have the asymptotic expansion (see \cite{cho2015inhomogeneous}, Lemma 3.5):
\begin{equation}\label{jdsmpt}
J_{\lambda}(r)=
\sqrt{\tfrac{2}{\pi }}\cos{(r-\lambda_{*})} r^{-\frac{1}{2}}-
\tfrac{(\lambda-\frac{1}{2})
\Gamma{(\lambda+\frac{3}{2})}}{
\sqrt{2\pi}\Gamma{(\lambda+\frac{1}{2})}}\sin{(r-\lambda_{*})}
r^{-\frac{3}{2}}
+h_{\lambda}(r),
\end{equation}
$r>1$, with $\lambda_{*}=\frac{\pi}{2}\lambda+\frac{\pi}{4}$, where the remainder $h_{\lambda}$ and its first derivative $h^{\prime}_{\lambda}$ satisfy the bounds
\begin{equation*}
|h_{\lambda}(r)|\lesssim r^{-\frac{5}{2}},\quad |h^{\prime}_{\lambda}(r)|\lesssim (1+r^{-1})r^{-\frac{5}{2}}.
\end{equation*}
For more on the basic properties of Bessel functions of the first kind, see e.g. \cite{grafakos2008classical}, Appendix B,
\cite{stein1993harmonic}, Chapter VIII, and the classic monograph \cite{watson1922treatise}. \par\medskip  For our purposes, we need the following extension of the asymptotic formula (\ref{jdsmpt}):
\begin{lemma}\label{bsllma}
Let $\operatorname{Re}\lambda>-1/2$. Then, for $r>1$, the Bessel function $J_{\lambda}(r)$
has the asymptotic expansion:
\begin{equation}\label{bssl}
J_{\lambda}(r)=
\sum_{0\leq \ell \leq M}
c_{\ell}(\lambda)\cos{(r+{\pi}\ell/{2}-\lambda_{*})}
r^{-(\ell+\frac{1}{2})}+L_{\lambda}(r;M),
\end{equation}
for each $M\geq 1$, where
\begin{equation*}
c_{\ell}(\lambda):=
\frac{\left(\lambda-\frac{1}{2}\right)_{\!\ell}}{\sqrt{2\pi} 2^{\ell-1}\ell !}
\frac{\Gamma{(\lambda+\ell+\frac{1}{2})}}{
\Gamma{(\lambda+\frac{1}{2})}},
\end{equation*}
and $r\mapsto L_{\lambda}(r;M)$ is a continuous function on $(1,\infty)$ such that
\begin{equation}\label{lmrnt}
\left|L_{\lambda}(r;M)\right|\lesssim
r^{-M-\frac{3}{2}}.
\end{equation}
\end{lemma}

Observe that the first term in the expansion
(\ref{bssl}) (the term that correspond to $\ell=0$) converges to $J_{-{1}/{2}}(r)$ when $\lambda\rightarrow-{1}/{2}$. One can also easily check that the expansion
(\ref{bssl}) reduces to (\ref{jdsmpt})
when $M=1$.
We give a proof of (\ref{bssl}), where we compute
$L_{\lambda}(r;M)$ explicitly.
\begin{proof}
We shall obtain the asymptotic expansion (\ref{bssl}) of the Bessel function $J_{\lambda}(r)$ on $r>1$. For $\operatorname{Re}\lambda>-1/2$, $J_{\lambda}$ has the integral representation (see e.g. \cite{grafakos2008classical,stein1993harmonic}):
\begin{equation}\label{jlmrs}
J_{\lambda}(r)={\nu(\lambda)}\, r^{\lambda}\int_{0}^{\infty}
e^{-rs}\left(\dot{\imath} e^{-\dot{\imath}r}(s^{2} +2\dot{\imath} s)^{\lambda-\frac{1}{2}} -\dot{\imath} e^{\dot{\imath}r} (s^{2} -2\dot{\imath} s)^{\lambda-\frac{1}{2}} \right)ds,
\end{equation}
for $r>1$, with the constant ${\nu(\lambda)}={2^{-\lambda}}/{\left(
\Gamma{(\tfrac{1}{2})}\Gamma{(\lambda+\tfrac{1}{2})}
\right)}$. Following \cite{cho2015inhomogeneous}, we write
\begin{equation}\label{bsslid}
\begin{split}
&\dot{\imath} e^{-\dot{\imath}r}(s^{2} +2\dot{\imath} s)^{\lambda-\frac{1}{2}} -
\dot{\imath} e^{\dot{\imath}r} (s^{2} -2\dot{\imath} s)^{\lambda-\frac{1}{2}} \\
&=(2s)^{\lambda-\frac{1}{2}} \left( e^{\dot{\imath} (r-{\lambda}_{*})}
\left(1+\tfrac{\dot{\imath}s}{2}\right)^{
\lambda-\frac{1}{2}}+e^{-\dot{\imath} (r-{\lambda}_{*})}\left(1-\tfrac{\dot{\imath}s}{2}
\right)^{\lambda-\frac{1}{2}}\right),
\end{split}
\end{equation}
where ${\lambda}_{*}$ denotes
$\tfrac{\pi}{2}\lambda+\tfrac{\pi}{4}$. Taking a finite number of terms of the Taylor series expansions of the factors $\left(1\pm\tfrac{\dot{\imath}s}{2}\right)^{\lambda-\frac{1}{2}}$ we get
\begin{equation}\label{tlr}
\left(1\pm\tfrac{\dot{\imath}s}{2}\right)^{\lambda-\frac{1}{2}}=
\sum_{\ell=0}^{M}
\Lambda^{\pm}_{\ell}(\lambda) s^{\ell}
+\Lambda^{\pm}_{M+1}
\left(1\pm\tfrac{\dot{\imath}
 s_{*}}{2}\right)^{\lambda-M-\frac{3}{2}}s^{M+1},
\end{equation}
for some $s_{*} \in (0,s)$, where
\begin{eqnarray*}
\Lambda^{+}_{\ell}(\lambda)&:=&
\frac{1}{2^{\ell}\ell !}{(\lambda-\tfrac{1}{2})_{\ell}}
\,e^{\dot{\imath}\ell\pi/2},\\
\Lambda^{-}_{\ell}(\lambda)&:=&\frac{
1}{2^{\ell}\ell !}{(\lambda-\tfrac{1}{2})_{\ell}}\,e^{-\dot{\imath}\ell\pi/2},\\
\end{eqnarray*}
with the falling factorial notation
\begin{equation*}
(\lambda-\tfrac{1}{2})_{\ell}:=
(\lambda-\tfrac{1}{2})\cdots (\lambda-\tfrac{1}{2}-\ell+1), \quad \ell\geq 1,
\qquad (\lambda-\tfrac{1}{2})_{0}=1.
\end{equation*}
Substitute for $\left(1\pm\tfrac{\dot{\imath}s}{2}\right)^{
\lambda-\frac{1}{2}}$ from (\ref{tlr})
into (\ref{bsslid}), then plug (\ref{bsslid})
into (\ref{jlmrs}) and use the fact that
\begin{equation*}
\int_{0}^{\infty}
e^{-rs}s^{\ell+\lambda-\frac{1}{2}}ds
={\Gamma{(\ell+\lambda+\tfrac{1}{2})}}{
r^{-(\ell+\lambda+\frac{1}{2})}},
\end{equation*}
to obtain
\begin{multline}\label{jlmrs1}
J_{\lambda}(r)=2^{\lambda-\frac{1}{2}}{\nu(\lambda)}
\sum_{\ell=0}^{M}
{\Gamma{(\ell+\lambda+\tfrac{1}{2})}}\\
\left(\Lambda^{+}_{\ell}(\lambda) e^{\dot{\imath} (r-{\lambda}_{*})}+\Lambda^{-}_{\ell}(\lambda) e^{-\dot{\imath} (r-{\lambda}_{*})} \right)
r^{-(\ell+\frac{1}{2})}
+L_{\lambda}(r;M),
\end{multline}
where
\begin{equation*}
L_{\lambda}(r;M):= 2^{\lambda-\frac{1}{2}}
{\nu(\lambda)} r^{\lambda}
\int_{0}^{\infty}e^{-r s}
\Sigma_{\lambda}(r;M)
 s^{M+\lambda+\frac{1}{2}}ds,
\end{equation*}
with
\begin{multline*}
\Sigma_{\lambda}(r;M):=\;
\Lambda^{+}_{M+1}(\lambda)
e^{\dot{\imath} (r-{\lambda}_{*})}
(1+{\dot{\imath}
 s_{*}}/{2})^{\lambda-M-\frac{3}{2}}\\+
\Lambda^{-}_{M+1}(\lambda)
e^{-\dot{\imath} (r-{\lambda}_{*})}
(1-{\dot{\imath}
 s_{*}}/{2})^{\lambda-M-\frac{3}{2}}.
\end{multline*}
Using the estimate
\begin{equation*}
\left|(1\pm{\dot{\imath}
 s_{*}}/{2})^{\lambda-M-\frac{3}{2}}\right|\lesssim
\left(1+s\right)^{\lambda-M-\frac{3}{2}},
\end{equation*}
we see that
\begin{equation*}
\left|L_{\lambda}(r;M)\right|\lesssim r^{\lambda}
\int_{0}^{\infty}e^{-r s}
\left(1+s\right)^{\lambda-M-\frac{3}{2}}
 s^{M+\lambda+\frac{1}{2}}ds,
\end{equation*}
from which follows the estimate
\begin{equation}\label{jlmrs2}
\left|L_{\lambda}(r;M)\right| \lesssim r^{-(M+\frac{3}{2})}.
\end{equation}
Indeed, if $\lambda\leq M+\frac{3}{2}$ then
${\left(1+{s}\right)^{\lambda-M-\frac{3}{2}}}
\lesssim 1$ and we have
\begin{equation*}
\left|L_{\lambda}(r;M)\right|\lesssim r^{\lambda}
\int_{0}^{\infty}e^{-r s}
 s^{M+\lambda+\frac{1}{2}}ds\approx r^{-(M-\frac{3}{2})}.
\end{equation*}
On the other hand, if $\lambda> M+\frac{3}{2}$ then
${\left(1+{s}\right)^{\lambda-M-\frac{3}{2}}}
\approx 1+s^{\lambda-M-\frac{3}{2}}$. In this case, the estimate (\ref{jlmrs2}) follows from the fact that
\begin{equation*}
r^{\lambda}
\int_{0}^{\infty}e^{-r s} s^{2\lambda-1}ds
\approx r^{-\lambda}=o(r^{-(M+\frac{3}{2})}).
\end{equation*}
Substituting for
$\nu(\lambda)$, $\Lambda^{+}_{\ell}(\lambda)$,
and $\Lambda^{-}_{\ell}(\lambda)$ into the expansion (\ref{jlmrs1}), and then using the simple fact that
\begin{equation*}
e^{\dot{\imath}(r+\frac{\pi}{2}\ell-\lambda_{*})}
+e^{-\dot{\imath}(r+\frac{\pi}{2}\ell-\lambda_{*})}=
2\cos{(r+{\pi}\ell/{2}-\lambda_{*})},
\end{equation*}
we obtain the asymptotic expansion (\ref{bssl}).\qed
\end{proof}
\section{Asymptotics in the Fourier space: A proof of Theorem \ref{thmsmt}}\label{prcs}
Applying formula (\ref{fjdl1}), we obtain the Fourier transform of $E_{\alpha,\beta}(e^{\dot{\imath} \varphi } |\cdot|^{\sigma})$ in the form of a radially symmetric oscillatory integral as follows:
\begin{eqnarray*}
\mathcal{F}\left(E_{\alpha,\beta}(e^{\dot{\imath} \varphi } |\cdot|^{\sigma})
\right)(\xi)
&=&(2\pi)^{\frac{n}{2}}|\xi|^{1-\frac{n}{2}}
\int_{0}^{\infty}
E_{\alpha,\beta}(e^{\dot{\imath} \varphi } r^{\sigma})
J_{\frac{n}{2}-1}( r |\xi|)
r^{\frac{n}{2}} dr\\
&=&(2\pi)^{\frac{n}{2}}|\xi|^{-n}
\int_{0}^{\infty}
E_{\alpha,\beta}(e^{\dot{\imath} \varphi } r^{\sigma}|\xi|^{-\sigma})
J_{\frac{n}{2}-1}( r )
r^{\frac{n}{2}} dr.
\end{eqnarray*}
Let $\phi$ be a smooth positive cut-off function supported in $[-2,2]$
such that $\phi(r)=1$ on $[-1,1]$ and let  $\psi:=1-\phi$. The asymptotic behaviour of the Bessel function
suggests splitting
\begin{equation}\label{frtj}
\mathcal{F}\left(E_{\alpha,\beta}(e^{\dot{\imath} \varphi } |\cdot|^{\sigma})
\right)(\xi)=(2\pi)^{n/2}|\xi|^{-n}
\left({\mathcal{M}_{\alpha,\beta,\varphi,\sigma}}(\xi)+
{\mathcal{N}_{\alpha,\beta,\varphi,\sigma}}(\xi)
\right),
\end{equation}
where
\begin{eqnarray*}
{\mathcal{M}_{\alpha,\beta,\varphi,\sigma}}(\xi)&:=&
\int_{0}^{\infty}
\phi(r)E_{\alpha,\beta}(e^{\dot{\imath} \varphi } r^{\sigma}|\xi|^{-\sigma})
\overline{J}_{n}(r)
dr, \\
{\mathcal{N}_{\alpha,\beta,\varphi,\sigma}}(\xi)&:=&
\int_{0}^{\infty}
\psi(r)
E_{\alpha,\beta}(e^{\dot{\imath} \varphi } r^{\sigma}|\xi|^{-\sigma})\overline{J}_{n}(r)dr,
\end{eqnarray*}
with
\begin{equation*}
\overline{J}_{n}(r):=J_{\frac{n}{2}-1}(r )
r^{\frac{n}{2}}.
\end{equation*}
\par\medskip
We examine the asymptotic behaviour of ${\mathcal{M}_{\alpha,\beta,\varphi,\sigma}}(\xi)$
and ${\mathcal{N}_{\alpha,\beta,\varphi,\sigma}}(\xi)$
both as $\xi\rightarrow 0$ and as $|\xi|\rightarrow +\infty$.
This is the purpose of Lemmas \ref{lmsmpt1} through \ref{lmsmpt2} below. In view of (\ref{frtj}), we see how Lemmas \ref{lmsmpt1} and \ref{lmsmpt2} together give us an insight of the asymptotic behaviour of $\mathcal{F}\left(E_{\alpha,\beta}(e^{\dot{\imath} \varphi } |\cdot|^{\sigma})\right)$, when $\sigma>(n-1)/2$.
Dominated convergence is the main tool used to estimate the remainder term in the asymptotic formulas (\ref{lzl})-(\ref{ltsmptnf}) of Lemma \ref{lmsmpt1}, and
the asymptotic formulas (\ref{htsmptz}) and (\ref{htsmptfty1}) of Lemma \ref{lmsmpt2}. When we speak of the \textit{remainder term} we are clearly assuming that the constant coefficient of the first term in the asymptotic formula is not equal to zero. This is not the case for all values of the parameters (see Remark \ref{rmcs1}
below Lemma \ref{lmsmpt1} and Remark \ref{rmcs2} below Lemma \ref{lmsmpt2}). If the constant coefficient of the first term is not zero, the asymptotic formula is in fact an asymptotic expansion, and the first term is the main term of the expansion. The integral representation (\ref{esln1}) of $E_{\alpha,\beta}$
is used to write ${\mathcal{M}_{\alpha,\beta,\varphi,\sigma}}$
and ${\mathcal{N}_{\alpha,\beta,\varphi,\sigma}}$ as double integrals that are easier to manipulate
in order to eventually justify passing the limit inside the integrals. Our asymptotic analysis
of ${\mathcal{M}_{\alpha,\beta,\varphi,\sigma}}$ exploits the fact that the argument of
$J_{{n}/{2}-1}$ is small enough, on the support of $\phi$, to employ formula (\ref{jdsmll1}), when $n>1$. If $n=1$, we use formula (\ref{jd1}).  \\
\indent The asymptotic analysis of ${\mathcal{N}_{\alpha,\beta,\varphi,\sigma}}$ is expectedly trickier. Bessel functions of large arguments oscillate rapidly and decay too slowly to apply dominated convergence directly here. Exploiting the asymptotic expansion (\ref{bssl}), one can write ${\mathcal{N}_{\alpha,\beta,\varphi,\sigma}}$ as a finite sum of oscillatory integrals. Roughly speaking, one can apply the dominated convergence theorem to each of these oscillatory integrals after integrating the Mittag-Leffler function by parts, sufficiently many times, against the oscillatory factor. This process is carried out in Lemma \ref{lmprts}.
Thanks to formula (\ref{esln1}),
the derivatives of $x\mapsto E_{\alpha,\beta}(e^{\dot{\imath} \varphi } |x|^{\sigma})$, $x\neq 0$, are obtained as contour integrals. We emphasize that the statements of the upcoming Lemmas \ref{lmsmpt1} through \ref{lmsmpt2} hold for all $\alpha \in (0,2)$ and $\beta>0$, provided that $|\varphi|>\alpha \pi/2$.
\begin{lemma}\label{lmsmpt1}
Let $\alpha\in (0,2)$,
$\beta>0$, and $\pi\alpha/2<|\varphi|\leq \pi$. Then
\begin{empheq}[left={
\mathcal{M}_{\alpha,\beta,\varphi,\sigma}(\xi)=\empheqlbrace}]
{alignat=2}
\label{lzl}&A_{\alpha,\beta,\varphi,\sigma}|\xi|^{\sigma}
+o(|\xi|^{\sigma}),&\quad \sigma<n;\vspace{0.5 cm}\\
\label{lzeq}
&B_{\alpha,\beta,\varphi,n}{|\xi|^{n}\log{|\xi|}}+O(|\xi|^{n}),&\quad \sigma=n;\vspace{0.5 cm}\\
\label{lzgeq}
&C_{\alpha,\beta,\varphi,\sigma}|\xi|^{n}+
o(|\xi|^{n}),&\quad \sigma>n,
\end{empheq}
as $\xi \rightarrow 0$, where
\begin{eqnarray*}
A_{\alpha,\beta,\varphi,\sigma}&:=&
-\frac{e^{-\dot{\imath} \varphi}
}{\Gamma{(\beta-\alpha)}}\int_{0}^{\infty}
r^{-\sigma}\phi(r)\overline{J}_{n}(r)dr,  \\
B_{\alpha,\beta,\varphi,n}&:=&
\frac{2^{1-\frac{n}{2}}}{\Gamma{(\frac{n}{2})}} \frac{ e^{-\dot{\imath}\varphi}}{\Gamma(\beta-\alpha)} \\
C_{\alpha,\beta,\varphi,\sigma}&:=& \frac{2^{1-\frac{n}{2}}}{\Gamma{(\frac{n}{2})}}\int_{0}^{\infty}
r^{n-1}E_{\alpha,\beta}(e^{\dot{\imath} \varphi } r^{\sigma})dr.
\end{eqnarray*}
The integrals that define the constants
$A_{\alpha,\beta,\varphi,\sigma}$ and
$C_{\alpha,\beta,\varphi,\sigma}$ are absolutely convergent when $\sigma<n$ and $\sigma>n$, respectively. Moreover, we have
\begin{equation}\label{ltsmptnf}
{\mathcal{M}_{
\alpha,\beta,\varphi,\sigma}}(\xi)= \frac{1}{\Gamma{(\beta)}}\int_{0}^{\infty}
\phi(r)\overline{J}_{n}(r)dr+
O(|\xi|^{-\sigma}),
\end{equation}
as $|\xi|\rightarrow+\infty$.
\end{lemma}
\begin{remark}\label{rmcs1}
Observe that, due to the factor $1/\Gamma{(\beta-\alpha)}$, the constant coefficients
$A_{\alpha,\beta,\varphi,\sigma}$ and $B_{\alpha,\beta,\varphi,\sigma}$ are equal to zero,
when $\beta-\alpha$ is a nonpositive integer. In that case, the asymptotic formulas
(\ref{lzl})  and (\ref{lzeq}) should be understood in the following sense:
\begin{equation*}
{\mathcal{M}_{\alpha,\beta,\varphi,\sigma}(\xi)}=
\begin{cases}
o(|\xi|^{\sigma}),& \sigma<n;\vspace{0.25 cm}\\
O(|\xi|^{n}),& \sigma=n,
\end{cases}
\end{equation*}
as $\xi \rightarrow 0$. This occurs if
$\alpha=\beta=1$  for instance, in which case $E_{\alpha,\beta}$
coincides with the exponential function.

\end{remark}
\begin{proof}
Let us begin with the proof of (\ref{ltsmptnf}).
We start with the well-known recurrence relation (see e.g.
\cite{gorenflo2020mittag},  Section 4.2, \cite{diethelm2002analysis}, Theorem 4.2, and \cite{gupta2007some}, formula (5))
\begin{equation*}
E_{\alpha,\beta}(z)=\frac{1}{\Gamma{(\beta)}}+z
E_{\alpha,\alpha+\beta}(z),\quad z\in \mathbb{C},
\end{equation*}
which is true for all $\alpha,\beta>0$. In particular, we have
\begin{equation}\label{q1mft}
E_{\alpha,\beta}(e^{\dot{\imath} \varphi } r^{\sigma}|\xi|^{-\sigma})=
\frac{1}{\Gamma{(\beta)}}+\frac{1}{|\xi|^{\sigma}}
e^{\dot{\imath} \varphi } r^{\sigma}
E_{\alpha,\alpha+\beta}(e^{\dot{\imath} \varphi } r^{\sigma}|\xi|^{-\sigma}),
\end{equation}
for all $r\geq 0$ and all $\xi\neq 0$. Since $|\varphi|>\alpha \pi/2$, a direct application of the inequality (\ref{ptws0}) implies that
\begin{equation}\label{ptws}
|E_{\alpha,\alpha+\beta}(e^{\dot{\imath} \varphi } r^{\sigma}|\xi|^{-\sigma})|\leq
C(1+r^{\sigma}|\xi|^{-\sigma})^{-1}
\leq C,
\end{equation}
for all $r>0$, and all $\xi\neq 0$.
Moreover, it follows from
formula (\ref{jdsmll1}) that, for any $n>1$, we have
\begin{equation}\label{jnr1}
\overline{J}_{n}(r) =\left(a_{n} +O(r) \right)
r^{n-1},
\end{equation}
which yields the estimate
\begin{equation}\label{zrst2}
\left|\overline{J}_{n}(r)\right|\lesssim
r^{n-1},\qquad r\in \operatorname{supp}\phi \cap [0,\infty).
\end{equation}
The estimate (\ref{zrst2}) holds true when $n=1$ as well, by
the identity (\ref{jd1}).
Combining the inequalities (\ref{ptws}) and (\ref{zrst2}), we see that
\begin{equation*}
r^{\sigma}\phi(r) | E_{\alpha,\alpha+\beta}(e^{\dot{\imath} \varphi } r^{\sigma}|\xi|^{-\sigma})
\overline{J}_{n}(r)|\lesssim
\phi(r)r^{n+\sigma-1},\quad r\geq 0,
\end{equation*}
uniformly in $\xi$. Since
$r\mapsto \phi(r)r^{n+\sigma-1}$ is an $L^{1}([0,\infty))$ function, one can apply the dominated convergence theorem to obtain
\begin{multline*}
\lim_{|\xi|\rightarrow+\infty}\int_{0}^{\infty}
r^{\sigma}\phi(r)E_{\alpha,\alpha+\beta}(e^{\dot{\imath} \varphi } r^{\sigma}|\xi|^{-\sigma})
\overline{J}_{n}(r)dr\\
= E_{\alpha,\alpha+\beta}(0)\int_{0}^{\infty}
r^{\sigma}\phi(r)\overline{J}_{n}(r)dr\\
= \frac{1}{\Gamma{(\alpha+\beta)}}\int_{0}^{\infty}
r^{\sigma}\phi(r)\overline{J}_{n}(r)dr.
\end{multline*}
The asymptotic formula (\ref{ltsmptnf}) follows then from (\ref{q1mft}) and the latter limit. \par\medskip
We turn to the asymptotic behaviour of $\mathcal{M}_{\alpha,\beta,\varphi,\sigma}$ as $\xi\rightarrow 0$.
Suppose that $\sigma<n$, and let $|\xi|<1$. Write
\begin{equation*}
\frac{|\xi|^{-\sigma}}{r^{\sigma} |\xi|^{-\sigma} e^{\dot{\imath} \varphi }-z}=\frac{e^{-\dot{\imath} \varphi }}{r^{\sigma}} +
\frac{z e^{-\dot{\imath} \varphi }}{r^{\sigma}(r^{\sigma} |\xi|^{-\sigma} e^{\dot{\imath} \varphi }-z)}.
\end{equation*}
Therefore, using the contour integral representation (\ref{esln1}), we have
\begin{multline*}
|\xi|^{-\sigma} E_{\alpha,\beta}(e^{\dot{\imath} \varphi } r^{\sigma}|\xi|^{-\sigma})
=\frac{1}{2\pi \dot{\imath} \alpha} |\xi|^{-\sigma}\int_{H_{\omega}}
\frac{e^{z^{1/\alpha}}{
z^{(1-\beta)/\alpha}}}{
 z-e^{\dot{\imath} \varphi } \,r^{\sigma}|\xi|^{-\sigma}}dz\\
= - \frac{e^{-\dot{\imath} \varphi}r^{-\sigma}}{2\pi\alpha \dot{\imath}}\int_{H_{\omega}}
e^{z^{1/\alpha}}
z^{(1-\beta)/\alpha} dz+ \frac{e^{-\dot{\imath} \varphi}r^{-\sigma}}{2\pi\alpha \dot{\imath}}\int_{H_{\omega}}
\frac{e^{z^{1/\alpha}}{
z^{(1+\alpha-\beta)/\alpha}}}{
 z-e^{\dot{\imath} \varphi } \,r^{\sigma}|\xi|^{-\sigma}}dz \\
=-\frac{e^{-\dot{\imath} \varphi}
r^{-\sigma}}{\Gamma{(\beta-\alpha)}}
+ e^{-\dot{\imath} \varphi}r^{-\sigma}E_{\alpha,\beta-\alpha}(e^{\dot{\imath} \varphi } \,r^{\sigma}|\xi|^{-\sigma}).
\end{multline*}
Multiply through by $|\xi|^{\sigma}\phi(r)\overline{J}_{n}(r)$ and then integrate in
$r$ to obtain
\begin{multline}\label{sln1}
{\mathcal{M}_{\alpha,\beta,\varphi,\sigma}}(\xi)=
-\frac{e^{-\dot{\imath} \varphi}
}{\Gamma{(\beta-\alpha)}}|\xi|^{\sigma}
\int_{0}^{\infty}
r^{-\sigma}\phi(r)\overline{J}_{n}(r)dr\\
+e^{-\dot{\imath} \varphi}
|\xi|^{\sigma}
\int_{0}^{\infty}
r^{-\sigma}\phi(r)\overline{J}_{n}(r)E_{\alpha,\beta-\alpha}(e^{\dot{\imath} \varphi } \,r^{\sigma}|\xi|^{-\sigma})dr.
\end{multline}
Observe that, by (\ref{zrst2}), we have
\begin{equation*}
r^{-\sigma}\phi(r)|\overline{J}_{n}(r)| \lesssim r^{n-\sigma-1}\phi(r),
\end{equation*}
for all $r>0$. We also have  $r^{n-\sigma-1}\phi\in L^{1}([0,\infty))$ when $\sigma<n$. Furthermore, similarly to (\ref{ptws}), there exists a positive constant $C$ such that
\begin{equation*}
|E_{\alpha,\beta-\alpha}(e^{\dot{\imath} \varphi } r^{\sigma}|\xi|^{-\sigma})|\leq \frac{C|\xi|^{\sigma}}{
|\xi|^{\sigma}+r^{\sigma}}\leq C,
\end{equation*}
for all $r\geq 0$, and all $\xi\in \mathbb{R}^{n}$.
It therefore follows by dominated convergence that
\begin{equation*}
\lim_{\xi\rightarrow 0}\int_{0}^{\infty}
r^{-\sigma}\phi(r)\overline{J}_{n}(r)E_{\alpha,\beta-\alpha}(e^{\dot{\imath} \varphi } r^{\sigma}|\xi|^{-\sigma})dr=0,
\end{equation*}
when $\sigma<n$. The asymptotic formula (\ref{lzl}) follows from (\ref{sln1}) together with the latter limit.\par\medskip
Next, we prove (\ref{lzeq}). Let $\sigma=n$, and let $|\xi|<1$.
Split
\begin{equation}\label{sptab} {\mathcal{M}_{\alpha,\beta,\varphi,n}}=
\mathcal{M}^{-}_{\alpha,\beta,\varphi,n}+
\mathcal{M}^{+}_{\alpha,\beta,\varphi,n},
\end{equation}
where
\begin{eqnarray*}
\mathcal{M}^{-}_{\alpha,\beta,\varphi,n}(\xi)&:=&
\int_{0}^{|\xi|}\phi(r) \overline{J}_{n}(r) E_{\alpha,\beta}(e^{\dot{\imath} \varphi } r^{n}|\xi|^{-n})\,dr,\\
\mathcal{M}^{+}_{\alpha,\beta,\varphi,n}(\xi)&:=&
\int_{|\xi|}^{2}\phi(r) \overline{J}_{n}(r) E_{\alpha,\beta}(e^{\dot{\imath} \varphi } r^{n}|\xi|^{-n})\,dr.
\end{eqnarray*}
Substituting for
$\overline{J}_{n}$ from (\ref{jnr1}), then changing variables $r\rightarrow r|\xi|$, we get
\begin{equation*}
|\xi|^{-n}\mathcal{M}^{-}_{\alpha,\beta,\varphi,n}(\xi)=
\int_{0}^{1}\phi(r|\xi|)
\left(a_{n} +O(r|\xi|) \right)
r^{n-1}
E_{\alpha,\beta}(e^{\dot{\imath} \varphi } r^{n})\,dr.
\end{equation*}
Since $r\mapsto E_{\alpha,\beta}(e^{\dot{\imath} \varphi } r^{n})$ is uniformly continuous on $[0,1]$, for all $n\geq 1$, and since
\begin{equation*}
\phi(r|\xi|)
\left|a_{n} +O(r|\xi|) \right|\lesssim 1,
\end{equation*}
for all $0\leq r\leq 1$ and all $\xi$ with $|\xi|<1$, we may use dominated convergence to
compute
\begin{equation}\label{axi0}
\lim_{\xi\rightarrow 0}|\xi|^{-n}\mathcal{M}^{-}_{\alpha,\beta,\varphi,n}(\xi)=
a_{n} \int_{0}^{1}r^{n-1}
E_{\alpha,\beta}(e^{\dot{\imath} \varphi } r^{n})\,dr,
\end{equation}
a finite number for all $n\geq 1$.
It remains to determine the asymptotic behaviour of
$\mathcal{M}^{+}_{\alpha,\beta,\varphi,n}(\xi)$ as $\xi\rightarrow 0$. Substituting for
$\overline{J}_{n}$ from (\ref{jnr1}) and
for $E_{\alpha,\beta}$ from (\ref{esln1}) into
the definition of ${\mathcal{M}^{+}_{\alpha,\beta,\varphi,n}}$, we see that
\begin{equation*}
\mathcal{M}^{+}_{\alpha,\beta,\varphi,n}(\xi)=
\frac{1}{2\pi \dot{\imath} \alpha}
\int_{|\xi|}^{2}\int_{H_{\omega}}\phi(r)
\left(a_{n} +O(r) \right) r^{n-1}
\frac{e^{z^{1/\alpha}}{
z^{(1-\beta)/\alpha}}}{
z-e^{\dot{\imath} \varphi }
\,r^{n}|\xi|^{-n}}dz\,dr.
\end{equation*}
It is important to note here that the constant involved in the $O$-notation in the last integral is independent of $z$. Using Fubini's theorem, we can write
\begin{equation}\label{bfbn}
\mathcal{M}^{+}_{\alpha,\beta,\varphi,n}(\xi)=\frac{a_{n}}{2\pi \dot{\imath} \alpha}
\int_{H_{\omega}}{e^{z^{1/\alpha}}{
z^{(1-\beta)/\alpha}}} \int_{|\xi|}^{2}\frac{\phi(r)
\left(1+O(r) \right)}{
 z-e^{\dot{\imath} \varphi }
\,r^{n}|\xi|^{-n}}r^{n-1}\,dr\,dz.
\end{equation}
The use of
Fubini's theorem here can be justified as follows: parameterize the contour $H_{\omega}$ to write
\begin{multline}\label{se1}
\int_{H_{\omega}}
\frac{e^{z^{1/\alpha}}{
z^{(1-\beta)/\alpha}}}{
z-r e^{\dot{\imath} \varphi }}dz=
\alpha
\sum_{\pm}\pm
e^{\pm\dot{\imath}
\frac{\omega}{\alpha}(1-\beta)}
\int_{1}^{\infty}
\frac{e^{e^{\pm\dot{\imath}
\frac{\omega}{\alpha}}{\rho}}
\,{\rho}^{\alpha-\beta}}{
{\rho}^{\alpha}-e^{\dot{\imath} (\varphi\mp\omega) }r}d{\rho}\\
+\dot{\imath}\int_{-\omega}^{\omega}
\frac{e^{ e^{\dot{\imath}{\theta}/\alpha}}\,
e^{\dot{\imath}{\theta}(1-\beta)/\alpha}}{
1-e^{\dot{\imath} (\varphi-{\theta})}r} d{\theta},\quad
r\geq 0,
\end{multline}
upon a change of variables.
The sum before the first integral on the right side of (\ref{se1}) is to be understood as the sum of the two terms that correspond to the
two indicated combinations of the $+$ and $-$ signs.
Then, using the inequalities (\ref{zrst11}) and (\ref{zrst12}) that imply, respectively, the uniform estimates
\begin{equation*}
\min_{\rho\geq 1}{|\rho^{\alpha}-e^{\dot{\imath} (\varphi\pm \omega )}
\,r^{n}|\xi|^{-n}|}\geq \varphi_{0},\quad
\min_{|\theta|\geq \omega}{|1-e^{\dot{\imath} (\varphi-\theta )}\,r^{n}|\xi|^{-n}|}\geq \varphi_{0},
\end{equation*}
we see that
\begin{equation*}
\int_{H_{\omega}}
\left|\frac{e^{z^{1/\alpha}}{
z^{(1-\beta)/\alpha}}}{
z-e^{\dot{\imath} \varphi }
\,r^{n}|\xi|^{-n}}\right||dz|\lesssim
2\alpha
\int_{1}^{\infty}
e^{\cos{(\omega/\alpha)}\rho}
\,{\rho}^{\alpha-\beta}d{\rho}+\int_{-\omega}^{\omega}
e^{\cos{(\theta/\alpha)}} d{\theta}.
\end{equation*}
Observe that
$(r,\rho)\mapsto \phi(r) r^{n-1} e^{\rho\cos{({\omega}/{\alpha})
}}{{\rho}^{\alpha-\beta}}$ is an $L^{1}([0,\infty)\times [1,\infty))$ function, and
$(r,\theta)\mapsto \phi(r) r^{n-1} e^{\cos{({\theta}/{\alpha})}}$
is clearly in $L^{1}([0,\infty)\times [-\omega,\omega])$.\par\medskip
Further split the inner integral in (\ref{bfbn}) into the sum of the two integrals
\begin{eqnarray*}
I_{1}(z,\xi) &:=& \int_{|\xi|}^{2}
\frac{r^{n-1}}{ z-e^{\dot{\imath} \varphi } \,r^{n}|\xi|^{-n}}dr, \\
I_{2}(z,\xi)&:=& |\xi|^{n}
\int_{|\xi|}^{2}(\phi(r)-1+\phi(r)O(r))
\frac{r^{n-1}}{|\xi|^{n}z-e^{\dot{\imath} \varphi } \,r^{n}}dr.
\end{eqnarray*}
Consider first the integral $I_{1}$. Changing variables $|\xi|^{-n}r^{n}\rightarrow r$ yields
\begin{eqnarray*}
I_{1}(z,\xi) &=& \frac{1}{n}|\xi|^{n}
\int_{1}^{2^{n}|\xi|^{-n}}
\frac{dr}{ z-e^{\dot{\imath} \varphi } \,r}\\
   &=&  -\frac{e^{-\dot{\imath} \varphi } }{n}|\xi|^{n}
\int_{1}^{2^{n}|\xi|^{-n}}
\frac{dr}{r-e^{-\dot{\imath} \varphi } \,z}.
\end{eqnarray*}
Using the identity
\begin{equation*}
\frac{1}{r-e^{-\dot{\imath} \varphi}z}=
\frac{1}{r}
+\frac{e^{-\dot{\imath} \varphi}z}{r(r-e^{-\dot{\imath} \varphi}z)},
\end{equation*}
we see that
\begin{equation*}
I_{1}(z,\xi) = -\frac{e^{-\dot{\imath} \varphi } }{n}|\xi|^{n}
\int_{1}^{2^{n}|\xi|^{-n}}\left(
\frac{1}{r}
+\frac{e^{-\dot{\imath} \varphi}z}{r(r-e^{-\dot{\imath} \varphi}z)}\right)dr.
\end{equation*}
A straightforward calculation gives
\begin{equation*}
\int_{1}^{2^{n}|\xi|^{-n}}
\frac{dr}{r(r-e^{-\dot{\imath} \varphi}z)}=
-\frac{e^{\dot{\imath} \varphi}}{z}\left(
\log{(1-e^{-\dot{\imath} \varphi}z)}-
\log{(1-e^{-\dot{\imath} \varphi}2^{-n}|\xi|^{n}z)}\right).
\end{equation*}
And since $\int_{1}^{2^{n}|\xi|^{-n}}
r^{-1}dr=\log{2^{n}}-n\log{|\xi|}$, we have
\begin{multline}\label{i1sgqln}
I_{1}(z,\xi) = -\frac{e^{-\dot{\imath} \varphi } }{n}
\left(-n|\xi|^{n}\log{|\xi|}+\log{\left(\frac{2^{n}}{1-e^{-\dot{\imath} \varphi}z}\right)}|\xi|^{n}\right.\\ +
|\xi|^{n}\log{\left(1-2^{-n}e^{-\dot{\imath} \varphi} z|\xi|^{n}\right)}\bigg).
\end{multline}
It is not difficult to verify that
\begin{equation}\label{ntlgx0}
\left|\log{\left(1-2^{-n}e^{-\dot{\imath} \varphi} z|\xi|^{n}\right)}\right| \lesssim
|z|^{2} |\xi|^{n},
\end{equation}
for all $z\in H_{\omega}$.
One way to see this is to write
\begin{equation*}
\log{\left(1-2^{-n}e^{-\dot{\imath} \varphi} z|\xi|^{n}\right)}=\int_{1} ^{1-2^{-n}e^{-\dot{\imath} \varphi} z|\xi|^{n}}   \frac{d w}{w},
\end{equation*}
where the integral is taken along the line segment connecting the endpoints in the complex plane. So, by the estimate (\ref{zrst1}) of Lemma \ref{infrt}, we have the bound
\begin{eqnarray*}
\left|\log{\left(1-2^{-n}e^{-\dot{\imath} \varphi} z|\xi|^{n}\right)}\right|&\leq &\max_{0\leq t\leq 1}{\frac{2^{-n} |z| |\xi|^{n}}{{|1-2^{-n}e^{-\dot{\imath} \varphi} z|\xi|^{n} t|}}} \\   &= &\frac{2^{-n} |z|^{2} |\xi|^{n}}{\min_{0\leq t\leq 1}{|\overline{z}-2^{-n}e^{-\dot{\imath} \varphi} |z|^{2}|\xi|^{n} t|}}\\
&= &\frac{2^{-n} |z|^{2} |\xi|^{n}}{\min_{0\leq t\leq 1}{|z-2^{-n}e^{\dot{\imath} \varphi} |z|^{2}|\xi|^{n} t|}}\\
&\leq &\frac{|z|^{2} |\xi|^{n}}{2^{n} \varphi_{0}},
\end{eqnarray*}
for all $z\in H_{\omega}$. Since $\rho\mapsto e^{\rho\cos{\left({\omega}/{
\alpha}\right)
}}{\rho^{\gamma}}$ is absolutely integrable on $[1,\infty)$ for all $\gamma\in \mathbb{R}$, the estimate (\ref{ntlgx0}) implies
\begin{multline*}
\left|
\int_{H_{\omega}}{e^{z^{1/\alpha}}{
z^{(1-\beta)/\alpha}}}\log{\left(1-2^{-n}e^{-\dot{\imath} \varphi} z|\xi|^{n}\right)}dz\right|\\
\begin{aligned}[b]
&\lesssim |\xi|^{n}
\int_{H_{\omega}}e^{|z|^{1/\alpha}\cos{\left({\arg{z}}/{\alpha}\right)
}}{|z|^{{(1-\beta)}/{\alpha}+2}}|dz|\\
&\leq |\xi|^{n}\left(
2\alpha \int_{1}^{\infty}e^{{\rho}
\cos{\left({\omega}/{\alpha}\right)
}}{\rho}^{3\alpha-\beta} d\rho+
\int_{-\omega}^{\omega}e^{\cos{({\theta}/{\alpha})}}d\theta\right)\\
&=O(|\xi|^{n}).
\end{aligned}
\end{multline*}
One can similarly show that the integral
\begin{equation*}
\widetilde{A}_{\alpha,\beta,\varphi,n}:=-\frac{e^{-\dot{\imath}\varphi}}{2 \pi \alpha n\dot{\imath}} \int_{H_{\omega}}{e^{z^{1/\alpha}}{
z^{(1-\beta)/\alpha}}}
\log\left(\frac{2^{n}}{1-
e^{-\dot{\imath}\varphi}z}\right)dz
\end{equation*}
is absolutely convergent. With these remarks in mind, it follows from (\ref{i1sgqln}) that
\begin{multline}\label{i1sn1}
\frac{1}{2 \pi \alpha \dot{\imath}}\int_{H_{\omega}}{e^{z^{1/\alpha}}{
z^{(1-\beta)/\alpha}}}I_{1}(z,\xi)dz\\
=\frac{
e^{-\dot{\imath}\varphi}}{\Gamma(\beta-\alpha)}
{|\xi|^{n}\log{|\xi|}}+ \widetilde{A}_{\alpha,\beta,\varphi,n}|\xi|^{n}
+o(|\xi|^{n}),
\end{multline}
as $\xi\rightarrow 0$. Next, we consider $I_{2}$. First observe that
\begin{equation*}
|\xi|^{-n}I_{2}(z,\xi)=
\int_{|\xi|}^{2}(\phi(r)-1+\phi(r)O(r))
\frac{r^{n-1}}{|\xi|^{n}z-e^{\dot{\imath} \varphi } \,r^{n}}dr.
\end{equation*}
Using the estimate (\ref{zrst1}), we see that
\begin{equation}
\min_{z\in H_{\omega}}{|z-e^{\dot{\imath} \varphi } r^{n}|\xi|^{-n} |}\geq \varphi_{0} r^{n}|\xi|^{-n}.
\end{equation}
Thus, we have
\begin{eqnarray*}
\left|\frac{r^{n-1}}{|\xi|^{n}z-e^{\dot{\imath} \varphi } \,r^{n}}\right|
&=&\frac{r^{n}|\xi|^{-n}}{|z-e^{\dot{\imath} \varphi } \,r^{n}|\xi|^{-n}|}
\frac{1}{r}
\\
&\lesssim& \frac{1}{r},\quad r>0,
\end{eqnarray*}
for all $z\in H_{\omega}$ and $\xi\in \mathbb{R}^{n}$.
Since $\operatorname{supp}
\left(\phi-1\right)\cap[0,2]
\subseteq[1,2]$, one has $\left(\phi-1\right)/r+\phi\in L^{1}([0,2])$. Thus, by dominated convergence, we conclude that
\begin{multline*}\label{i23o0}
\frac{1}{2\pi \dot{\imath} \alpha}\lim_{\xi\rightarrow 0}
|\xi|^{-n}\int_{H_{\omega}}{e^{z^{1/\alpha}}{
z^{(1-\beta)/\alpha}}}
I_{2}(z,\xi)\,dz\\=-\frac{e^{-\dot{\imath}\varphi}}{
\Gamma{(\beta-\alpha)}}\int_{0}^{2}
\left(\frac{\phi(r)-1}{r}+\phi(r)O(1)\right)dr.
\end{multline*}
Combining (\ref{i1sn1}) with this last limit, it follows from
(\ref{bfbn}) that
\begin{equation}
\mathcal{M}^{+}_{\alpha,\beta,\varphi,n}(\xi)=\frac{
a_{n} e^{-\dot{\imath}\varphi}}{\Gamma(\beta-\alpha)}
{|\xi|^{n}\log{|\xi|}}+O(|\xi|^{n}),
\end{equation}
as $\xi\rightarrow 0$.
Together with (\ref{sptab}) and
the limit (\ref{axi0}), this proves
(\ref{lzeq}). \par\medskip
Lastly, we prove (\ref{lzgeq}). Suppose that $\sigma>n$ and let $|\xi|<1$. Change variables $r\rightarrow |\xi|r$ and then
substitute for
$\overline{J}_{n}$ from (\ref{jnr1}). We obtain
\begin{multline}\label{sgtn0}
\mathcal{M}_{\alpha,\beta,\varphi,\sigma}(\xi)
=|\xi|^{n}
\int_{0}^{\infty}\phi(|\xi|r)
\left(a_{n} +O(|\xi|r) \right)r^{n-1}
E_{\alpha,\beta}(e^{\dot{\imath} \varphi } r^{\sigma})dr\\
=a_{n}|\xi|^{n}\int_{0}^{\infty}r^{n-1}
E_{\alpha,\beta}(e^{\dot{\imath} \varphi } r^{\sigma})dr+
a_{n}|\xi|^{n}\int_{0}^{\infty}
\left(\phi(|\xi|r)-1\right)r^{n-1}
E_{\alpha,\beta}(e^{\dot{\imath} \varphi } r^{\sigma})dr
\\+
|\xi|^{n}
\int_{0}^{\infty}\phi(|\xi|r)O(|\xi|r)r^{n-1} E_{\alpha,\beta}(e^{\dot{\imath} \varphi } r^{\sigma})dr.
\end{multline}
Using (\ref{ptws0}) once again, we see that
\begin{equation*}
r^{n-1}|E_{\alpha,\beta}(e^{\dot{\imath} \varphi } r^{\sigma})|\leq
\frac{Cr^{n-1} }{1+r^{\sigma}},\quad r\geq 0,
\end{equation*}
for some positive constant $C$ independent of $r$. Note here that $r^{n-1}/(1+r^{\sigma})$ is
an $L^{1}([0,\infty))$ function if $\sigma>n$. Since
$0\leq \phi\leq 1$ and $\operatorname{supp}\phi\cap[0,\infty)\subseteq [0,2]$, one has
\begin{eqnarray*}
\left|\phi(|\xi|r)-1\right|&\lesssim& 1,\\
\phi(|\xi|r)\left|O(|\xi|r)\right|&\lesssim& 1,
\end{eqnarray*}
for all $r\geq 0$, and all $\xi$ such that $|\xi|<1$. Consequently,
by dominated convergence, we have
\begin{eqnarray*}
\lim_{\xi\rightarrow 0}
\int_{0}^{\infty}
\left(\phi(|\xi|r)-1\right)r^{n-1}
E_{\alpha,\beta}(e^{\dot{\imath} \varphi } r^{\sigma})dr&=&0,\\
\lim_{\xi\rightarrow 0}\int_{0}^{\infty}\phi(|\xi|r)O(|\xi|r)r^{n-1} E_{\alpha,\beta}(e^{\dot{\imath} \varphi } r^{\sigma})dr&=&0,
\end{eqnarray*}
when $\sigma>n$. Together with the last equality in (\ref{sgtn0}), this shows (\ref{lzgeq}).
\qed
\end{proof}
Before we proceed to determine the asymptotic behaviour of $\mathcal{N}_{\alpha,\beta,\varphi,\sigma}$, we prove the following technical lemma.
\begin{lemma}\label{lmprts}
Let $\alpha\in (0,2)$,
$\beta>0$, and $\pi\alpha/2<|\varphi|\leq \pi$.
Fix $\xi \in\mathbb{R}^{n}\setminus \{0\}$ and a nonnegative integer $\ell$. For each nonnegative integer $m$, let
\begin{equation*}
{Q_{m}(r)}:=
\sum_{0\leq k\leq m}
\widetilde{C}_{k,m}(\varphi,\sigma)r^{k\sigma}\int_{H_{\omega}}
\frac{e^{z^{1/\alpha}}{
z^{(1-\beta)/\alpha}}}{
(z-e^{\dot{\imath} \varphi } \,r^{\sigma})^{k+1}}dz,
\end{equation*}
where
\begin{equation*}
\widetilde{C}_{k,m}(\varphi,\sigma):=
e^{\dot{\imath} \varphi k}\sum_{0\leq j\leq k}(-1)^{j}\binom{k}{j}
{((k-j)\sigma-m+1)}^{(m)},\;\;\; 0\leq k\leq m,
\end{equation*}
with the Pochhammer symbol $
{x}^{(m)}=x(x+1)\cdots(x+m-1)$, $x\in \mathbb{R}$. \par\medskip  If $\sigma>(n-1)/2$, then
\begin{multline}\label{prtsq}
\int_{0}^{\infty}
r^{\frac{n-1}{2}-\ell} \cos{(r)} \psi(r) \int_{H_{\omega}}
\frac{e^{z^{1/\alpha}}{
z^{(1-\beta)/\alpha}}}{
z-  e^{\dot{\imath} \varphi} r^{\sigma}|\xi|^{-\sigma}}dz\,dr\\
=\sum_{\ell_{1}+\ell_{2}+\ell_{3}=N}
b_{\ell_{1},\ell_{2},\ell_{3}}
\int_{0}^{\infty}
r^{\frac{n-1}{2}-\ell-N+\ell_{2}}S_{N}{(r)}\psi^{(\ell_{2})}(r)
Q_{\ell_{3}}(r|\xi|^{-1})\,dr,
\end{multline}
for each positive integer $N$\!, where
\begin{equation*}
S_{N}{(r)}:=
\begin{cases}
-\sin{r},&N=4k-3;\\
-\cos{r},&N=4k-2;\\
\sin{r},&N=4k-1;\\
\cos{r},&N=4k,\;k\in \mathbb{N},
\end{cases}
\end{equation*}
and
\begin{equation*}
b_{\ell_{1},\ell_{2},\ell_{3}}:=
\tfrac{N!}{\ell_{1}!\,\ell_{2}!\,\ell_{3}!}
(\tfrac{n-1}{2}-\ell)_{\ell_{1}},
\end{equation*}
with the falling factorial notation
$(x)_{\ell_{1}}=x
(x-1)\cdots(x-\ell_{1}+1)$, $x\in \mathbb{R}$.
\end{lemma}
\begin{proof}
Let $\xi \in\mathbb{R}^{n}\setminus \{0\}$ and $j\geq 1$ be an integer. Then
\begin{eqnarray}
\nonumber \lim_{r\rightarrow \infty}
\frac{r^{j\sigma}}{|\xi|^{j\sigma}}\left|\int_{H_{\omega}}
\frac{e^{z^{1/\alpha}}{
z^{(1-\beta)/\alpha}}}{(
z-  e^{\dot{\imath} \varphi} r^{\sigma}|\xi|^{-\sigma})^{j}}dz\right|
&=&\left|\int_{H_{\omega}}
{e^{z^{1/\alpha}}{
z^{(1-\beta)/\alpha}}}dz\right|\\
\label{bndr1}&=&
\frac{1}{|\Gamma{(\beta-\alpha)}|}.
\end{eqnarray}
This follows by dominated convergence. To see that,
start by parameterizing the contour $H_{\omega}$ as in
(\ref{se1}) to obtain
\begin{multline}\label{prmtsm}
\int_{H_{\omega}}
\frac{e^{z^{1/\alpha}}{
z^{(1-\beta)/\alpha}}}{(
z-  e^{\dot{\imath} \varphi} r^{\sigma}|\xi|^{-\sigma})^{j}}dz
\\=\alpha
\sum_{\pm}\pm
e^{\pm\dot{\imath}
\frac{\omega}{\alpha}(1-\beta+(1-j)\alpha)}
\int_{1}^{\infty}
\frac{e^{e^{\pm\dot{\imath}
\frac{\omega}{\alpha}}{\rho}}
\,{\rho}^{\alpha-\beta}}{{(
\rho^{\alpha}-e^{\dot{\imath} (\varphi\mp\omega) } r^{\sigma}|\xi|^{-\sigma})^{j}}}d{\rho}\\
+\dot{\imath}
\int_{-\omega}^{\omega}
\frac{e^{ e^{\dot{\imath}{\theta}/\alpha}}\,
e^{\dot{\imath}{\frac{\theta}{\alpha}}(1-\beta+(1-j)\alpha)}}{
(1-e^{\dot{\imath} (\varphi-{\theta})}r^{\sigma}|\xi|^{-\sigma})^{j}} d{\theta}.
\end{multline}
Using the estimate (\ref{zrst11}) of Lemma \ref{infrt}, we find
\begin{eqnarray*}
\frac{r^{j\sigma}}{|\xi|^{j\sigma}}\left|\frac{e^{e^{\pm\dot{\imath}
\frac{\omega}{\alpha}}{\rho}}
\,{\rho}^{\alpha-\beta}}{{(
\rho^{\alpha}-e^{\dot{\imath} (\varphi\mp\omega) } r^{\sigma}|\xi|^{-\sigma})^{j}}}\right|
&\leq&\frac{e^{\rho
\cos{\left({\omega}/{\alpha}\right)
}}{{\rho}^{\alpha-\beta}}r^{j\sigma}|\xi|^{-j\sigma}}{
\varphi_{0}^{j}\max{\{r^{j\sigma}|\xi|^{-j\sigma},4^{j}\}}}\\
&\lesssim& e^{\rho
\cos{\left({\omega}/{\alpha}\right)
}}{{\rho}^{\alpha-\beta}},
\end{eqnarray*}
for all $r\geq0$ and $\rho\geq 1$. Similarly, by the estimate (\ref{zrst12}) of Lemma \ref{infrt}, we have
\begin{eqnarray*}
\frac{r^{j\sigma}}{|\xi|^{j\sigma}}\left|
\frac{e^{ e^{\dot{\imath}{\theta}/\alpha}}\,
e^{\dot{\imath}{\frac{\theta}{\alpha}}(1-\beta+(1-j)\alpha)}}{
(1-e^{\dot{\imath} (\varphi-{\theta})}r^{\sigma}|\xi|^{-\sigma})^{j}}\right|
&\leq&\frac{e^{
\cos{\left({\theta}/{\alpha}\right)
}}r^{j\sigma}|\xi|^{-j\sigma}}{
\varphi_{0}^{j}\max{\{r^{j\sigma}|\xi|^{-j\sigma},4^{j}\}}}\\
&\lesssim& e^{
\cos{\left({\theta}/{\alpha}\right)
}},
\end{eqnarray*}
for all $r\geq0$ and $|\theta|\leq \omega$.
As noted in the proof of Lemma \ref{lmsmpt1}, $\rho\mapsto e^{\rho
\cos{\left({\omega}/{\alpha}\right)
}}{{\rho}^{\gamma}}$ is in $L^{1}([1,\infty))$, for any
$\gamma\in \mathbb{R}$, and  $\theta\mapsto e^{\cos{\left({\theta}/{\alpha}\right)
}}$ is clearly in $L^{1}([-\omega,\omega])$.
This justifies applying dominated convergence.
Since $r\mapsto E_{\alpha,\beta}(e^{\dot{\imath} \varphi } r)$
is analytic, the function
\begin{equation*}
r\mapsto \int_{H_{\omega}}\frac{e^{z^{1/\alpha}}{
z^{(1-\beta)/\alpha}}}{z-e^{\dot{\imath} \varphi } \,r^{\sigma}}dz
\end{equation*}
is in fact smooth on the support of $\psi$. Using
Fa\`{a} di Bruno's formula (see e.g. \cite{comtet2012advanced}) for the derivative of a composition of smooth functions, we find
\begin{multline}\label{drvnt1}
\partial^{m}_{r}\int_{H_{\omega}}\frac{e^{z^{1/\alpha}}{
z^{(1-\beta)/\alpha}}}{z-e^{\dot{\imath} \varphi } \,r^{\sigma}|\xi|^{-\sigma}}dz=\int_{H_{\omega}}
e^{z^{1/\alpha}}{
z^{(1-\beta)/\alpha}}\,\partial^{m}_{r}\left(
\frac{1}{{z-e^{\dot{\imath} \varphi } \,r^{\sigma}|\xi|^{-\sigma}}}\right)dz\\
=\sum_{0\leq k\leq m}
\widetilde{C}_{k,m}(\varphi,\sigma) \frac{r^{k\sigma-m}}
{|\xi|^{k\sigma}}
\displaystyle\int_{H_{\omega}}
\frac{e^{z^{1/\alpha}}{
z^{(1-\beta)/\alpha}}}{
(z-e^{\dot{\imath} \varphi } \,r^{\sigma}|\xi|^{-\sigma})^{k+1}}dz\\
=r^{-m}Q_{m}(r/|\xi|),\quad\;
r\in \operatorname{supp}\psi.
\end{multline}
The interchange of the derivative and the integral here can be justified similarly to the limit (\ref{bndr1}).
Using (\ref{drvnt1}), we find that, for all integers $\ell\geq 0$ and $N\geq 1$, we have
\begin{multline}\label{nthdrv}
\partial^{N}_{r}\left(
r^{\frac{n-1}{2}-\ell}\psi(r)
\int_{H_{\omega}}\frac{e^{z^{1/\alpha}}{
z^{(1-\beta)/\alpha}}}{z-e^{\dot{\imath} \varphi } \,r^{\sigma}|\xi|^{-\sigma}}dz\right)\\
\begin{aligned}[b]
&=\sum_{\ell_{1}+\ell_{2}+\ell_{3}=N}
\tfrac{N!}{\ell_{1}!\,\ell_{2}!\,\ell_{3}!}\,
\partial^{\ell_{1}}_{r}\!\!\left(r^{\frac{n-1}{2}-\ell}\right)
\psi^{(\ell_{2})}(r)\,\partial^{\ell_{3}}_{r}\!\!
\int_{H_{\omega}}\frac{e^{z^{1/\alpha}}{
z^{(1-\beta)/\alpha}}}{z-e^{\dot{\imath} \varphi } \,r^{\sigma}|\xi|^{-\sigma}}dz\\
&=\sum_{\ell_{1}+\ell_{2}+\ell_{3}=N}
\tfrac{N!}{\ell_{1}!\,\ell_{2}!\,\ell_{3}!}
\left(\tfrac{n-1}{2}-\ell\right)_{\ell_{1}}\,
r^{\frac{n-1}{2}-\ell-\ell_{1}-\ell_{3}}\psi^{(\ell_{2})}(r)
Q_{\ell_{3}}(r/|\xi|)\\
&=\sum_{\ell_{1}+\ell_{2}+\ell_{3}=N}
b_{\ell_{1},\ell_{2},\ell_{3}}
r^{\frac{n-1}{2}-\ell-N+\ell_{2}}\psi^{(\ell_{2})}(r)
Q_{\ell_{3}}(r/|\xi|).
\end{aligned}
\end{multline}
Formula (\ref{prtsq}) follows then by integration by parts $N$ times. The boundary terms vanish each time, provided that $\sigma>(n-1)/2$. To see this, it suffices by (\ref{nthdrv}) to show that
\begin{equation*}
\lim_{r\rightarrow \infty}\left|S_{N}{(r)}\sum_{\ell_{1}+\ell_{2}+\ell_{3}=m}
b_{\ell_{1},\ell_{2},\ell_{3}}
r^{\frac{n-1}{2}-\ell-m+\ell_{2}}\psi^{(\ell_{2})}(r)
Q_{\ell_{3}}(r/|\xi|)\right|=0,
\end{equation*}
for each integer $0\leq m\leq N$. Observe that
\begin{equation*}
\operatorname{supp}\psi\subseteq
[1,\infty),\qquad\operatorname{supp}\psi^{(j)}\subseteq
[1,2],\quad \forall j\geq 1.
\end{equation*}
Thus, for all nonnegative integers
$m_{1}$ and $m_{2}$, one has
\begin{equation*}
r^{\frac{n-1}{2}-\ell-m_{2}}
|\psi^{(m_{1})}(r)|\leq
r^{\frac{n-1}{2}}|\psi^{(m_{1})}(r)|.
\end{equation*}
Consequently, for each integer $0\leq m\leq N$, we have
\begin{multline}\label{bndrtrm}
\left|\sum_{\ell_{1}+\ell_{2}+\ell_{3}=m}
b_{\ell_{1},\ell_{2},\ell_{3}}
r^{\frac{n-1}{2}-\ell-m+\ell_{2}}
\psi^{(\ell_{2})}(r)
Q_{\ell_{3}}(r/|\xi|)\right|\\
\leq
\sum_{\ell_{1}+\ell_{2}+\ell_{3}=m}
|b_{\ell_{1},\ell_{2},\ell_{3}}|
r^{\frac{n-1}{2}}
|\psi^{(\ell_{2})}(r)|
|Q_{\ell_{3}}(r/|\xi|)|\\
\leq C_{m}\, r^{\frac{n-1}{2}}
\sum_{0\leq j\leq m}
|Q_{j}(r/|\xi|)|,
\end{multline}
where
\begin{equation*}
C_{m}=\max_{\ell_{1}+\ell_{2}+\ell_{3}=m}{
|b_{\ell_{1},\ell_{2},\ell_{3}}|}\max_{0\leq j\leq m}
\max_{0\leq r<\infty}|\psi^{(j)}(r)|.
\end{equation*}
We deduce that the left side of (\ref{bndrtrm})
is bounded by
\begin{equation*}
C_{m}r^{\frac{n-1}{2}-\sigma}|\xi|^{\sigma}
\sum_{1\leq k\leq m}\sum_{0\leq j\leq k}
|\widetilde{C}_{j,k}(\varphi,\sigma)|
\frac{r^{(j+1)\sigma}}{
|\xi|^{(j+1)\sigma}}\left|\int_{H_{\omega}}
\frac{e^{z^{1/\alpha}}{
z^{(1-\beta)/\alpha}}}{
(z-e^{\dot{\imath} \varphi } \,r^{\sigma}|\xi|^{-\sigma})^{j+1}}dz\right|,
\end{equation*}
for each integer $0\leq m\leq N$. The latter function vanishes as $r\rightarrow \infty$, by the limit (\ref{bndr1}) and the simple fact that
\begin{equation*}
\lim_{r\rightarrow \infty}r^{\frac{n-1}{2}-\sigma}=0,
\end{equation*}
when $\sigma>(n-1)/2$.\qed
\end{proof}
\begin{lemma}\label{lmsmpt2}
Let $\alpha\in (0,2)$,
$\beta>0$, and $\pi\alpha/2<|\varphi|\leq \pi$.
If $ \sigma>{(n-1)}/2$, then
\begin{equation}\label{htsmptz}
{\mathcal{N}_{\alpha,\beta,\varphi,\sigma}}(\xi)
=D_{\alpha,\beta,\varphi,\sigma}|\xi|^{\sigma}+
O(|\xi|^{2\sigma}),
\end{equation}
as $\xi\rightarrow 0$, with the constant
$D_{\alpha,\beta,\varphi,\sigma}$ given by (\ref{hjx1n}) below. Furthermore, we have
\begin{equation}\label{htsmptfty1}
\mathcal{N}_{\alpha,\beta,\varphi,\sigma}(\xi)=
\widetilde{D}_{\beta,n}
+o(1)
\end{equation}
as $|\xi| \rightarrow +\infty$, with the constant
$\widetilde{D}_{\beta,n}$ given explicitly in (\ref{dtldabc}).
\end{lemma}
\begin{remark}\label{rmcs2}
As can be seen from (\ref{hjx1n}) below, the constant coefficient
$D_{\alpha,\beta,\varphi,\sigma}$
of the first term
in the asymptotic formula (\ref{htsmptz}) is a finite linear combination of absolutely convergent integrals.
We must take into account
the possibility that these terms might cancel out.
Clearly, $D_{\alpha,\beta,\varphi,\sigma}=0$
if $\alpha=\beta$, because of the factor $1/\Gamma{(\beta-\alpha)}$. The asymptotic formula (\ref{htsmptz})
should be understood in the sense that
\begin{equation*}
{\mathcal{N}_{\alpha,\beta,\varphi,\sigma}}(\xi)
=O(|\xi|^{2\sigma}),\quad \xi\rightarrow 0,
\end{equation*}
if $D_{\alpha,\beta,\varphi,\sigma}=0$.
This applies as well to the constant $\widetilde{D}_{\beta,n}$
in the first term in (\ref{htsmptfty1}). The asymptotic formula (\ref{htsmptfty1}) is to be understood in the sense that
\begin{equation*}
\mathcal{N}_{\alpha,\beta,\varphi,\sigma}(\xi)=
o(1),\quad |\xi|\rightarrow \infty,
\end{equation*}
when the sum
(\ref{dtldabc}) equals zero.
\end{remark}
\begin{proof}
Use the contour integral representation (\ref{esln1}) to write
\begin{equation}\label{hx1}
{\mathcal{N}_{\alpha,\beta,\varphi,\sigma}}(\xi)
=\frac{1}{2\pi \dot{\imath} \alpha}
\int_{0}^{\infty}\psi(r)\overline{J}_{n}(r)
\int_{H_{\omega}}\frac{e^{z^{1/\alpha}}{
z^{(1-\beta)/\alpha}}}{z-e^{\dot{\imath} \varphi } \,r^{\sigma}|\xi|^{-\sigma}}dz\,dr.
\end{equation}
Let $n>1$ and take $M=\lceil \frac{n-1}{2}\rceil+1$.
This choice of $M$ will be explained momentarily.
Use the expansion (\ref{bssl}) of Lemma \ref{bsllma}
to substitute for $J_{{n}/{2}-1}(r)$ into the right side of (\ref{hx1}) to see that
\begin{multline}\label{hjx}
{\mathcal{N}_{\alpha,\beta,\varphi,\sigma}}(\xi)
=\frac{1}{2\pi \dot{\imath} \alpha}
\sum_{0\leq \ell\leq \lceil \frac{n-1}{2}\rceil+1}
c_{\ell}(\tfrac{n}{2}-1)\int_{0}^{\infty}
\cos{(r+{\pi}\ell/{2}-\lambda_{*})}\\
r^{\frac{n-1}{2}-\ell}\psi(r)
\int_{H_{\omega}}\frac{e^{z^{1/\alpha}}{
z^{(1-\beta)/\alpha}}}{z-e^{\dot{\imath} \varphi } \,r^{\sigma}|\xi|^{-\sigma}}dz\,dr
+{\widetilde{\mathcal{N}}_{\alpha,\beta,\varphi,\sigma}}(\xi),
\end{multline}
where
\begin{equation*}
{\widetilde{\mathcal{N}}_{\alpha,\beta,\varphi,\sigma}}(\xi)
:=\frac{1}{2\pi \dot{\imath} \alpha}
\int_{0}^{\infty}r^{\frac{n}{2}}
L_{\frac{n}{2}-1}( r;\lceil \tfrac{n-1}{2}\rceil+1) \psi (r)
\int_{H_{\omega}}\frac{e^{z^{1/\alpha}}{
z^{(1-\beta)/\alpha}}}{z-e^{\dot{\imath} \varphi } \,r^{\sigma}|\xi|^{-\sigma}}dz\,dr.
\end{equation*}
If $n=1$, use formula (\ref{jd1}) to substitute for $J_{-1/2}(r)$. In this case,
the right side of (\ref{hjx}) reduces to the first term (the term that correspond to $\ell=0$).
Applying formula (\ref{prtsq}) of Lemma \ref{lmprts} with $N=\lceil \tfrac{n-1}{2}\rceil+2$, we find that
\begin{multline}\label{hjx1}
{\mathcal{N}_{\alpha,\beta,\varphi,\sigma}}(\xi)
=\frac{1}{2\pi \dot{\imath} \alpha}
\sum_{0\leq \ell\leq \lceil \frac{n-1}{2}\rceil+1}
c_{\ell}(\tfrac{n}{2}-1)\;
\sum_{\ell_{1}+\ell_{2}+\ell_{3}=N}
b_{\ell_{1},\ell_{2},\ell_{3}}\,\\
\int_{0}^{\infty}
r^{N_{\ell}+\ell_{2}}S_{\ell,N}{(r)}\psi^{(\ell_{2})}(r)
Q_{\ell_{3}}(r/|\xi|)dr
+{\widetilde{\mathcal{N}}_{\alpha,\beta,\varphi,\sigma}}(\xi),
\end{multline}
where
\begin{equation*}
S_{\ell,N}{(r)}:=
S_{N}\left(r+\frac{\pi}{2}\ell-\frac{\pi}{4}n+\frac{\pi}{4}\right),
\quad \ell\geq 0,
\end{equation*}
and $N_{\ell}:={(n-1)}/{2}-\ell-N$.
Notice that $N_{\ell}\leq-2$ for all $\ell\geq 0$.
Let us fix $0\leq
\ell\leq \lceil \frac{n-1}{2}\rceil+1$ and $0\leq \ell_{2}, \ell_{3}\leq N$. Substituting for $Q_{\ell_{3}}$
from its definition in Lemma \ref{lmprts}, the integral
\begin{equation}\label{ntl}
\int_{0}^{\infty}
r^{N_{\ell}+\ell_{2}}S_{\ell,N}{(r)}
\psi^{(\ell_{2})}(r)
Q_{\ell_{3}}(r/|\xi|)dr=\sum_{0\leq k\leq \ell_{3}}
\widetilde{C}_{k,\ell_{3}}(\varphi,\sigma)\mathcal{Q}_{k}(\xi),
\end{equation}
where
\begin{multline*}
\mathcal{Q}_{k}(\xi):={|\xi|^{-k\sigma}}\int_{0}^{\infty}
r^{N_{\ell}+\ell_{2}+k \sigma}
S_{\ell,N}{(r)}
\psi^{(\ell_{2})}(r)\\
\int_{H_{\omega}}
\frac{e^{z^{1/\alpha}}{
z^{(1-\beta)/\alpha}}}{
(z-e^{\dot{\imath} \varphi } \,{r^{\sigma}}{|\xi|^{-\sigma}})^{k+1}}dz\,dr,\quad
0\leq k\leq \ell_{3}.
\end{multline*}
 \par\medskip
First we prove (\ref{htsmptz}). Use the identity
\begin{equation*}
\frac{1}{z-e^{\dot{\imath} \varphi } \,{r^{\sigma}}{|\xi|^{-\sigma}}}
= -e^{-\dot{\imath}\varphi}\frac{|\xi|^{\sigma}}{r^{\sigma}}
\left(1+\frac{|\xi|^{\sigma} z}{e^{\dot{\imath} \varphi } \,r^{\sigma}-|\xi|^{\sigma} z}\right),
\end{equation*}
then apply the binomial theorem to get
\begin{multline}\label{nthw1}
\int_{H_{\omega}}
\frac{e^{z^{1/\alpha}}{
z^{(1-\beta)/\alpha}}}{
(z-e^{\dot{\imath} \varphi } \,{r^{\sigma}}{|\xi|^{-\sigma}})^{k+1}}dz
=(-1)^{k+1}
e^{-\dot{\imath}(k+1)\varphi}\frac{|\xi|^{(k+1)\sigma}}{r^{(k+1)\sigma}}
\\ \hfill \left(\int_{H_{\omega}}
e^{z^{1/\alpha}}{
z^{(1-\beta)/\alpha}}\,dz+
\sum_{j=1}^{k+1}\binom{k+1}{j}|\xi|^{j\sigma}
\int_{H_{\omega}}\frac{
e^{z^{1/\alpha}}{
z^{(1-\beta)/\alpha+j}}}{(e^{\dot{\imath} \varphi } \,r^{\sigma}-|\xi|^{\sigma} z)^{j}}dz\right).
\end{multline}
An adaptation of the estimate (\ref{zrst1}) of Lemma \ref{infrt} implies that
\begin{equation*}
\min_{r\geq 1,\,\xi\in\mathbb{R}^{n}\!,\,z\in H_{\omega}}{||\xi|^{\sigma} z -r e^{\dot{\imath} \varphi }|}\gtrsim 1.
\end{equation*}
Employing this estimate, we see that
\begin{multline*}
\left|\int_{H_{\omega}}\frac{
e^{z^{1/\alpha}}{
z^{(1-\beta)/\alpha+j}}}{(e^{\dot{\imath} \varphi } \,r^{\sigma}-|\xi|^{\sigma} z)^{j}}dz\right|\lesssim
\int_{H_{\omega}}e^{|z|^{1/\alpha}\cos{\left({\arg{z}}/{\alpha}\right)
}}{|z|^{{(1-\beta)}/{\alpha}+j}}|dz|\\
\leq
2\alpha\int_{1}^{\infty}e^{{\rho}
\cos{\left({\omega}/{\alpha}\right)
}}{\rho}^{(1+j)\alpha-\beta} d\rho+
\int_{-\omega}^{\omega}e^{\cos{({\theta}/{\alpha})}}d\theta,
\end{multline*}
for all $r\geq 1$ and any $\xi\in \mathbb{R}^{n}$.
The right side of the latter inequality exists and is finite for every $j\geq 1$, as demonstrated in the proof of Lemma \ref{lmsmpt1}.
It follows then from (\ref{nthw1}) that
\begin{multline}\label{hwz1}
\frac{|\xi|^{-k\sigma}}{2\pi \dot{\imath} \alpha}\int_{H_{\omega}}
\frac{e^{z^{1/\alpha}}{
z^{(1-\beta)/\alpha}}}{
(z-e^{\dot{\imath} \varphi } \,{r^{\sigma}}{|\xi|^{-\sigma}})^{k+1}}dz\\
=\left(\frac{(-1)^{k+1}
e^{-\dot{\imath}(k+1)\varphi}}{\Gamma{(\beta-\alpha)}}
|\xi|^{\sigma}+
O(|\xi|^{2\sigma})\right)r^{-(k+1)\sigma},
\end{multline}
as $\xi\rightarrow 0$, for all $r\geq 1$. The implicit constant in the $O$-notation in (\ref{hwz1}) is independent of $r$. Furthermore, since $N_{\ell}<-1$, for
every $\ell\geq 0$, and $\operatorname{supp}\psi^{(\ell_{2})}\subseteq
[1,2]$ for all $\ell_{2}>0$, we have that
$r\mapsto r^{N_{\ell}+\ell_{2}-\sigma}\psi^{(\ell_{2})}(r)$ is
an $L^{1}([0,\infty))$ function.
Hence, using (\ref{hwz1}), we deduce that
\begin{multline*}
\frac{1}{2\pi \dot{\imath}\alpha}\mathcal{Q}_{k}(\xi)=\frac{(-1)^{k+1}
e^{-\dot{\imath}(k+1)\varphi}}{\Gamma{(\beta-\alpha)}}
|\xi|^{\sigma} \int_{0}^{\infty}
r^{N_{\ell}+\ell_{2}-\sigma}
S_{\ell,N}{(r)}
\psi^{(\ell_{2})}(r)dr\\ +
O(|\xi|^{2\sigma}),
\end{multline*}
as $\xi\rightarrow 0$. By the estimate (\ref{lmrnt}), we have
\begin{equation*}
r^{\frac{n}{2}}\left|L_{\frac{n}{2}-1}(r;\lceil \tfrac{n-1}{2}\rceil+1)\right|\lesssim
r^{\frac{n-1}{2}-\lceil \frac{n-1}{2}\rceil-2}\leq
r^{-2},\qquad r\geq 1.
\end{equation*}
Thus, $r^{\frac{n}{2}}L_{{n}/{2}-1}( r;\lceil \tfrac{n-1}{2}\rceil+1) \psi (r)$ is absolutely integrable on $[0,\infty[$, and using (\ref{hwz1}) with $k=0$  yields
\begin{equation*}
{\widetilde{\mathcal{N}}_{\alpha,\beta,\varphi,\sigma}}(\xi)
=\frac{-
e^{-\dot{\imath}\varphi}}{\Gamma{(\beta-\alpha)}}
|\xi|^{\sigma}
\int_{0}^{\infty}r^{\frac{n}{2}-\sigma}
L_{\frac{n}{2}-1}( r;\lceil \tfrac{n-1}{2}\rceil+1) \psi (r)\,dr+
O(|\xi|^{2\sigma}),
\end{equation*}
as $\xi\rightarrow 0$. Substitute for $\mathcal{Q}_{k}$ into (\ref{ntl}), and then insert the result into the sum on the right side of (\ref{hjx1}). Substitute also for
${\widetilde{\mathcal{N}}_{\alpha,\beta,\varphi,\sigma}}$
from the latter formula into (\ref{hjx1}). We obtain
(\ref{htsmptz}) with the constant
\begin{multline}\label{hjx1n}
D_{\alpha,\beta,\varphi,\sigma}
=\bigg(
\sum_{0\leq \ell\leq \lceil \frac{n-1}{2}\rceil+1} c_{\ell}(\tfrac{n}{2}-1)\;
\sum_{\ell_{1}+\ell_{2}+\ell_{3}=\lceil \frac{n-1}{2}\rceil+2}
b_{\ell_{1},\ell_{2},\ell_{3}}\\
\left.\sum_{0\leq k\leq \ell_{3}}(-1)^{k+1}
e^{-\dot{\imath}(k+1)\varphi}
\widetilde{C}_{k,\ell_{3}}(\varphi,\sigma)\right.\\
\left.
 \int_{0}^{\infty}
r^{\frac{n-1}{2}-\lceil \frac{n-1}{2}\rceil-2-\ell+\ell_{2}-\sigma}
S_{\ell,\lceil \tfrac{n-1}{2}\rceil+2}{(r)}
\psi^{(\ell_{2})}(r)dr\right.\\
-
e^{-\dot{\imath}\varphi}
\int_{0}^{\infty}r^{\frac{n}{2}-\sigma}
L_{\frac{n}{2}-1}( r;\lceil \tfrac{n-1}{2}\rceil+1)\, \psi (r)\,dr\bigg)\frac{1}{\Gamma{(\beta-\alpha)}},
\end{multline}
having chosen $N=\lceil \tfrac{n-1}{2}\rceil+2$. \par\medskip
We turn our attention to
the asymptotic behaviour of
${\mathcal{N}_{\alpha,\beta,\varphi,\sigma}}(\xi)$ when $|\xi| \rightarrow \infty$. The proof of (\ref{htsmptfty1}) is based on the same idea as that of (\ref{htsmptz}), but the details are a little more involved. The starting point is formula (\ref{hjx1})
along with the integral (\ref{ntl}).
Observe that
$r^{N_{\ell}+\ell_{2}}\psi^{(\ell_{2})} \in L^{1}([0,\infty))$, for every $\ell,\ell_{2}\geq 0$, if $N=\lceil \tfrac{n-1}{2}\rceil+2$.
By (\ref{q1mft}), one has
\begin{multline}\label{q00}
\frac{1}{2\pi \alpha \dot{\imath}}\mathcal{Q}_{0}(\xi)=\int_{0}^{\infty}
r^{N_{\ell}+\ell_{2}}S_{\ell,N}{(r)} \psi^{(\ell_{2})}(r)
E_{\alpha,\beta}(e^{\dot{\imath} \varphi } r^{\sigma}|\xi|^{-\sigma})\,dr\\
=\frac{1}{\Gamma{(\beta)}}\int_{0}^{\infty}
r^{N_{\ell}+\ell_{2}}S_{\ell,N}{(r)} \psi^{(\ell_{2})}(r)\,dr
 \\+
e^{\dot{\imath} \varphi }\int_{0}^{\infty}
r^{N_{\ell}+\ell_{2}}S_{\ell,N}{(r)} \psi^{(\ell_{2})}(r)
r^{\sigma}|\xi|^{-\sigma}
E_{\alpha,\alpha+\beta}(e^{\dot{\imath} \varphi } r^{\sigma}|\xi|^{-\sigma})\,dr.
\end{multline}
Using (\ref{ptws0}), we have the uniform estimate
\begin{equation}\label{unfrmest}
r^{\sigma}|\xi|^{-\sigma}
\left|E_{\alpha,\alpha+\beta}(e^{\dot{\imath} \varphi } r^{\sigma}|\xi|^{-\sigma})\right|\leq
\frac{C r^{\sigma}|\xi|^{-\sigma}}{1+r^{\sigma}|\xi|^{-\sigma}}\leq C,
\end{equation}
for some positive number $C$. By dominated convergence, the last integral on the right side of (\ref{q00}) is convergent to zero as $|\xi|\rightarrow \infty$. In other words,
\begin{equation}\label{q01}
\frac{1}{2\pi \alpha \dot{\imath}}\mathcal{Q}_{0}(\xi)
=\frac{1}{\Gamma{(\beta)}}\int_{0}^{\infty}
r^{N_{\ell}+\ell_{2}} S_{\ell,N}{(r)}\psi^{(\ell_{2})}(r)\,dr
+o(1),\quad |\xi|\rightarrow \infty.
\end{equation}
The integral ${\widetilde{\mathcal{N}}_{\alpha,\beta,\varphi,\sigma}}(\xi)$
can be treated similarly to $\mathcal{Q}_{0}(\xi)$.
We have
\begin{multline}\label{ntldfn}
{\widetilde{\mathcal{N}}_{\alpha,\beta,\varphi,\sigma}}(\xi)=
\int_{0}^{\infty}r^{\frac{n}{2}}
L_{\frac{n}{2}-1}( r;\lceil \tfrac{n-1}{2}\rceil+1) \psi (r)
E_{\alpha,\beta}(e^{\dot{\imath} \varphi } r^{\sigma}|\xi|^{-\sigma})\,dr\\
=\frac{1}{\Gamma{(\beta)}}\int_{0}^{\infty}
r^{\frac{n}{2}}
L_{\frac{n}{2}-1}( r;\lceil \tfrac{n-1}{2}\rceil+1) \psi (r)
\,dr
 \\+
e^{\dot{\imath} \varphi }\int_{0}^{\infty}
r^{\frac{n}{2}}
L_{\frac{n}{2}-1}( r;\lceil \tfrac{n-1}{2}\rceil+1) \psi (r)
r^{\sigma}|\xi|^{-\sigma}
E_{\alpha,\alpha+\beta}(e^{\dot{\imath} \varphi } r^{\sigma}|\xi|^{-\sigma})\,dr\\
=\frac{1}{\Gamma{(\beta)}}\int_{0}^{\infty}
r^{\frac{n}{2}}
L_{\frac{n}{2}-1}( r;\lceil \tfrac{n-1}{2}\rceil+1) \psi (r)
\,dr+o(1),
\end{multline}
as $|\xi|\rightarrow \infty$,
by the uniform estimate (\ref{unfrmest}) and absolute integrability of
$r^{\frac{n}{2}}
L_{{n}/{2}-1}( r;\lceil {(n-1)}/{2}\rceil+1) \psi$ on $[0,\infty)$.
We are going to show that
\begin{equation}\label{qk10}
\lim_{|\xi|\rightarrow \infty}
\mathcal{Q}_{k}(\xi)=0,
\end{equation}
for all $k\geq 1$. Combined with (\ref{q01})
and (\ref{ntldfn}), that would conclude the proof of (\ref{htsmptfty1}) with
\begin{multline}\label{dtldabc}
\widetilde{D}_{\beta,n}:=
\bigg(
\sum_{0\leq \ell\leq \lceil \frac{n-1}{2}\rceil+1}c_{\ell}(\tfrac{n}{2}-1)\;
\sum_{\ell_{1}+\ell_{2}+\ell_{3}=N}
b_{\ell_{1},\ell_{2},\ell_{3}}\\
\int_{0}^{\infty}
r^{\frac{n-1}{2}-\lceil \frac{n-1}{2}\rceil-2-\ell+\ell_{2}}
S_{\ell,\lceil \tfrac{n-1}{2}\rceil+2}{(r)}\psi^{(\ell_{2})}(r)\,dr
\\
+\int_{0}^{\infty}
r^{\frac{n}{2}}
L_{\frac{n}{2}-1}( r;\lceil \tfrac{n-1}{2}\rceil+1) \psi (r)
\,dr\bigg)\frac{1}{\Gamma{(\beta)}}.
\end{multline}
Notice that
\begin{equation*}
\widetilde{C}_{0,\ell_{3}}(\varphi,\sigma)=
\begin{cases}
1,&\ell_{3}=0;
\vspace{0.15 cm}\\
0,&\ell_{3}\geq 1.
\end{cases}
\end{equation*}
It remains to prove (\ref{qk10}). Recall that
\begin{equation*}
\mathcal{Q}_{k}(\xi)=\int_{0}^{\infty}
r^{N_{\ell}+\ell_{2}}S_{\ell,N}{(r)}\psi^{(\ell_{2})}(r)
\int_{H_{\omega}}
\frac{e^{z^{1/\alpha}}{
z^{(1-\beta)/\alpha}}{r^{k\sigma}|\xi|^{-k\sigma}}}{
(z-e^{\dot{\imath} \varphi } \,{r^{\sigma}}{|\xi|^{-\sigma}})^{k+1}}dz\,dr.
\end{equation*}
Upon parameterizing the contour $H_{\omega}$ similarly to (\ref{prmtsm}), we see that
\begin{equation*}
|\mathcal{Q}_{k}(\xi)|\leq \alpha
(\mathcal{A}^{+}_{k}(\xi)+\mathcal{A}^{-}_{k}(\xi))+\mathcal{B}_{k}(\xi),
\end{equation*}
where
\begin{eqnarray*}
\mathcal{A}^{\pm}_{k}(\xi)&:=&
\int_{0}^{\infty}
\int_{1}^{\infty}
r^{N_{\ell}+\ell_{2}}|\psi^{(\ell_{2})}(r)|
\frac{e^{\rho\cos{(\omega/\alpha)}}
\,{\rho}^{\alpha-\beta} {r^{k\sigma}|\xi|^{-k\sigma}}}{|
\rho^{\alpha}-e^{\dot{\imath} (\varphi\pm\omega) } r^{\sigma}|\xi|^{-\sigma}|^{k+1}}d{\rho}dr,\\
\mathcal{B}_{k}(\xi)&:=&\int_{0}^{\infty}
\int_{-\omega}^{\omega}
r^{N_{\ell}+\ell_{2}}|\psi^{(\ell_{2})}(r)|
\frac{e^{ \cos{(\theta/\alpha)}}{r^{k\sigma}|\xi|^{-k\sigma}}}{
|1-e^{\dot{\imath} (\varphi-{\theta})}r^{\sigma}|\xi|^{-\sigma}|^{k+1}} d{\theta}dr.
\end{eqnarray*}
Since $\operatorname{supp}\phi\subseteq [0,2]$, $\operatorname{supp}\psi\subseteq [1,\infty)$, $0\leq \phi,\psi\leq 1$, and $\phi+\psi=1$, using the inequality (\ref{zrst11}) of Lemma \ref{infrt} one last time yields
\begin{align}\label{lstr1}
\nonumber &\frac{{r^{k\sigma}}{|\xi|^{-k\sigma}}}{|
\rho^{\alpha}-e^{\dot{\imath} (\varphi\pm\omega) } r^{\sigma}|\xi|^{-\sigma}|^{k+1}}\\
\nonumber &=\,\phi\left
({r}/{|\xi|}\right)\frac{{r^{k\sigma}}{|\xi|^{-k\sigma}}}{|
\rho^{\alpha}-e^{\dot{\imath} (\varphi\pm\omega) } r^{\sigma}|\xi|^{-\sigma}|^{k+1}}+
\,\psi\left
({r}/{|\xi|}\right)\frac{{r^{k\sigma}}{|\xi|^{-k\sigma}}}{|
\rho^{\alpha}-e^{\dot{\imath} (\varphi\pm\omega) } r^{\sigma}|\xi|^{-\sigma}|^{k+1}}\\
\nonumber &=
\frac{(r/|\xi|)^{k\sigma}\phi\left
({r}/{|\xi|}\right)}{|
\rho^{\alpha}-e^{\dot{\imath} (\varphi\pm\omega) } r^{\sigma}|\xi|^{-\sigma}|^{k+1}}+
\frac{\,\psi\left
({r}/{|\xi|}\right)}{(r/|\xi|)^{\sigma}}
\left(\frac{{r^{\sigma}}{|\xi|^{-\sigma}}}{|
\rho^{\alpha}-e^{\dot{\imath} (\varphi\pm\omega) } r^{\sigma}|\xi|^{-\sigma}|}
\right)^{k+1}\\
\nonumber &\leq \frac{1}{
\varphi_{0}^{k+1}}\frac{\phi\left
({r}/{|\xi|}\right)}{(|\xi|/r)^{k\sigma}}+
\frac{1}{
\varphi_{0}^{k+1}}\frac{\,\psi\left
({r}/{|\xi|}\right)}{(r/|\xi|)^{\sigma}}\\
&\leq \frac{2^{k\sigma}}{
\varphi_{0}^{k+1}}+\frac{1}{
\varphi_{0}^{k+1}},
\end{align}
for all $r\geq 0$, $\rho\geq 1$, and $\xi\neq 0$.
As an immediate consequence of the uniform estimate (\ref{lstr1}), the integrand in $\mathcal{A}^{\pm}_{k}(\xi)$
is bounded by a constant multiple of
\begin{equation*}
(r,\rho)\mapsto
r^{N_{\ell}+\ell_{2}}|\psi^{(\ell_{2})}(r)|
\,e^{\rho\cos{(\omega/\alpha)}}
\,{\rho}^{\alpha-\beta},
\end{equation*}
which is an $L^{1}([0,\infty)\times [1,\infty))$ function,
for each $\ell,\ell_{2}\geq 0$.
One can also use the inequality (\ref{zrst12})
of Lemma \ref{infrt} and proceed analogously to (\ref{lstr1}) to obtain
\begin{equation*}
\frac{{r^{k\sigma}|\xi|^{-k\sigma}}}{
|1-e^{\dot{\imath} (\varphi-{\theta})}r^{\sigma}|\xi|^{-\sigma}|^{k+1}}\lesssim 1,
\end{equation*}
for all $r\geq 0$, $-\omega\leq \theta\leq \omega$,
and $\xi\neq 0$. The integrand in $\mathcal{B}_{k}(\xi)$
is therefore bounded by a constant multiple of
\begin{equation*}
(r,\theta)\mapsto
r^{N_{\ell}+\ell_{2}}|\psi^{(\ell_{2})}(r)|
\,e^{\cos{(\theta/\alpha)}},
\end{equation*}
which is an $L^{1}([0,\infty)\times [-\omega,\omega])$ function,
for all $\ell,\ell_{2}\geq 0$. Finally, by dominated convergence,
we have
\begin{equation*}
\lim_{|\xi|\rightarrow \infty}
\mathcal{A}^{\pm}_{k}(\xi)=
\lim_{|\xi|\rightarrow \infty}
\mathcal{B}_{k}(\xi)=0.
\end{equation*}
This proves (\ref{qk10}). \qed
\end{proof}
\subsection{\textbf{Proof of Theorem \ref{thmsmt}}}
Theorem \ref{thmsmt} follows from
Lemma \ref{lmsmpt1} and Lemma \ref{lmsmpt2} together.
In (\ref{frtj}), we have written
\begin{equation*}
\mathcal{F}\left(E_{\alpha,\beta}(e^{\dot{\imath} \varphi } |\cdot|^{\sigma})
\right)(\xi)=(2\pi)^{n/2}|\xi|^{-n}
\left({\mathcal{M}_{\alpha,\beta,\varphi,\sigma}}(\xi)+
{\mathcal{N}_{\alpha,\beta,\varphi,\sigma}}(\xi)
\right).
\end{equation*}
The asymptotic formula (\ref{eab1}) follows then from (\ref{lzl})-(\ref{lzgeq}) of Lemma \ref{lmsmpt1} and (\ref{htsmptz}) of  Lemma \ref{lmsmpt2} combined.
Similarly, in view of (\ref{frtj}), formula (\ref{ltsmptnf}) of Lemma \ref{lmsmpt1} together with formula (\ref{htsmptfty1}) of Lemma \ref{lmsmpt2} give (\ref{eab2}).
\begin{remark}
The asymptotic behaviour of $\mathcal{F} E_{\alpha,\beta}(e^{\dot{\imath} \varphi} |\cdot|^{\sigma})(\xi)$
is given in Theorem \ref{thmsmt} in terms of
$O$-notation. The reason is that we must account for the possible cancellation that might occur upon adding an asymptotic expansion of
${\mathcal{M}_{\alpha,\beta,\varphi,\sigma}}(\xi)$ to that of
${\mathcal{N}_{\alpha,\beta,\varphi,\sigma}}(\xi)$ in accordance
with (\ref{frtj}). We also have to take into account the special cases where one or more of the coefficients $A_{\alpha,\beta,\varphi,\sigma}$,
$B_{\alpha,\beta,\varphi,\sigma}$, $D_{\alpha,\beta,\varphi,\sigma}$, or $\widetilde{D}_{\beta,n}$ is equal to zero, as discussed in Remark \ref{rmcs1} and Remark \ref{rmcs2}.
\end{remark}
\begin{remark}
Numerical tests presented in Appendix A below suggest that the asymptotic upper bound obtained in (\ref{eab1}) of Theorem \ref{thmsmt} is tight.
\end{remark}
\section*{Appendix A: Numerical tests}
We test the asymptotic formula (\ref{eab1}) of Theorem \ref{thmsmt} numerically for various values of the parameters $\alpha$, $\beta$, and $\sigma$, in the dimensions
$n=1,2$. We focus on the
case $(n-1)/2<\sigma\leq n$, where
$ E_{\alpha,\beta}(e^{\dot{\imath} \varphi} |\cdot|^{\sigma})$ is less regular and has slower decay.
\begin{figure}[H]
\centering
\begin{subfigure}{.48\textwidth}
  \centering
  \includegraphics[width=1\linewidth]{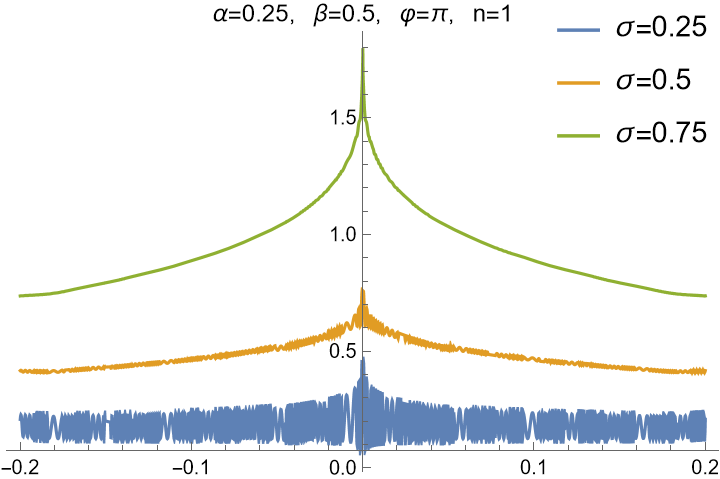}
\caption{$\alpha=0.25$, $\beta=0.5$}
\label{n111}
\end{subfigure}%
\hfill
\begin{subfigure}{.48\textwidth}
  \centering
  \includegraphics[width=1\linewidth]{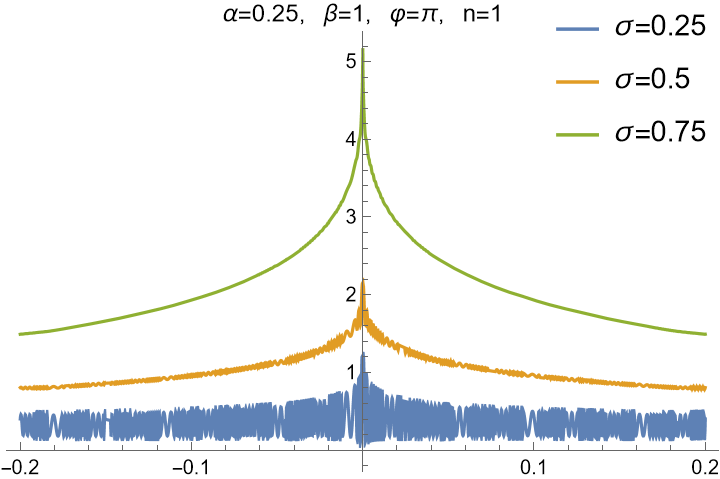}
\caption{$\alpha=0.25$, $\beta=1$}
\label{n112}
\end{subfigure}
\end{figure}
\unskip
\begin{figure}[H]
\ContinuedFloat
\centering
\begin{subfigure}{.48\textwidth}
\centering
  \includegraphics[width=1\linewidth]{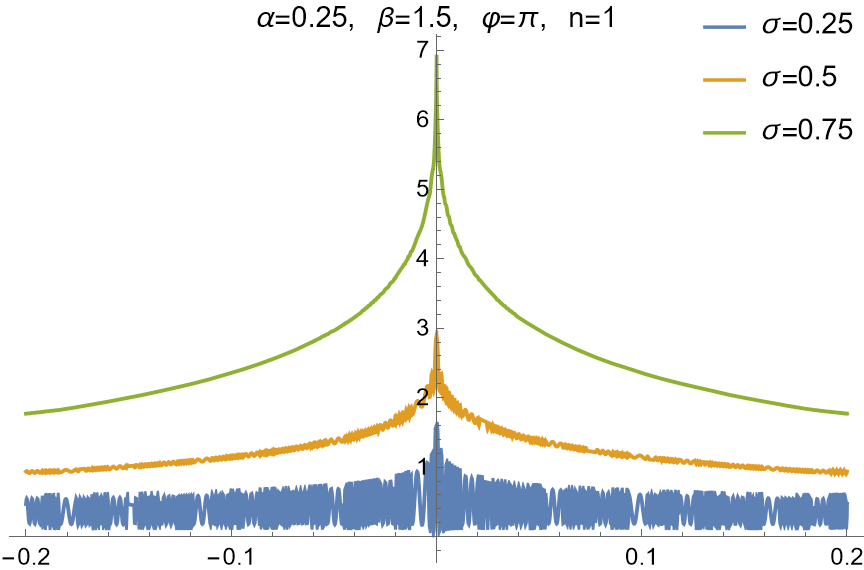}
\caption{$\alpha=0.25$, $\beta=1.5$}
\label{n113}
\end{subfigure}%
\hfill
\begin{subfigure}{.48\textwidth}
\centering
  \includegraphics[width=1\linewidth]{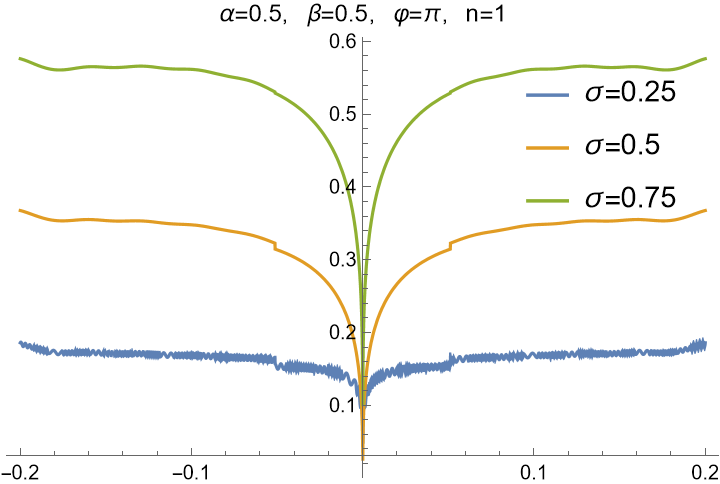}
\caption{$\alpha=0.5$, $\beta=0.5$}
\label{n121}
\end{subfigure}
\end{figure}
\unskip
\begin{figure}[H]
\ContinuedFloat
\centering
\begin{subfigure}{.48\textwidth}
\centering
  \includegraphics[width=1\linewidth]{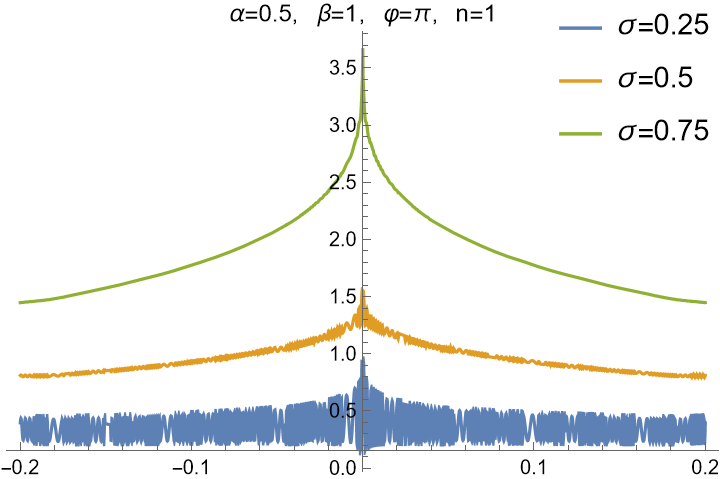}
\caption{$\alpha=0.5$, $\beta=1$}
\label{n122}
\end{subfigure}%
\hfill
\begin{subfigure}{.48\textwidth}
\centering
  \includegraphics[width=1\linewidth]{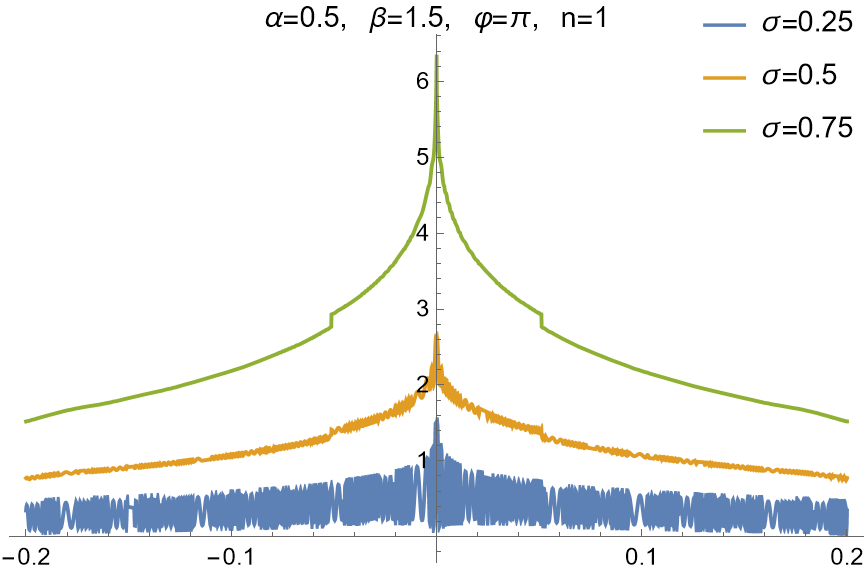}
\caption{$\alpha=0.5$, $\beta=1.5$}
\label{n123}
\end{subfigure}
\end{figure}
\unskip
\begin{figure}[H]
\ContinuedFloat
\centering
\begin{subfigure}{.48\textwidth}
\centering
  \includegraphics[width=1\linewidth]{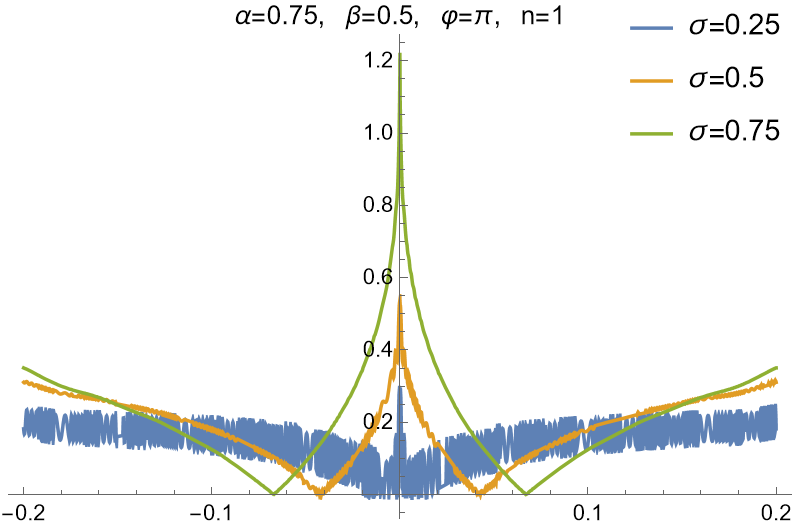}
\caption{$\alpha=0.75$, $\beta=0.5$}
\label{n131}
\end{subfigure}%
\hfill
\begin{subfigure}{.48\textwidth}
  \centering
  \includegraphics[width=1\linewidth]{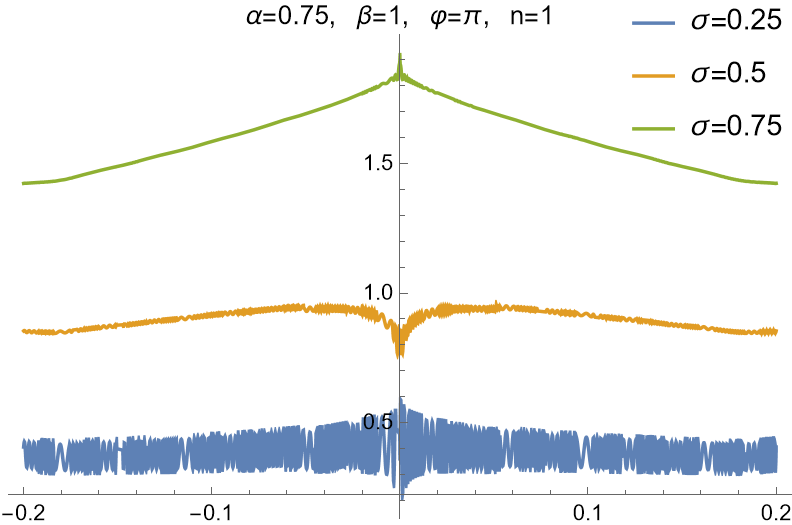}
\caption{$\alpha=0.75$, $\beta=1$}
\label{n132}
\end{subfigure}
\vspace{-1 cm}
\end{figure}
\unskip
\begin{figure}[H]
\ContinuedFloat
\centering
\begin{subfigure}{.48\textwidth}
\centering
  \includegraphics[width=1\linewidth]{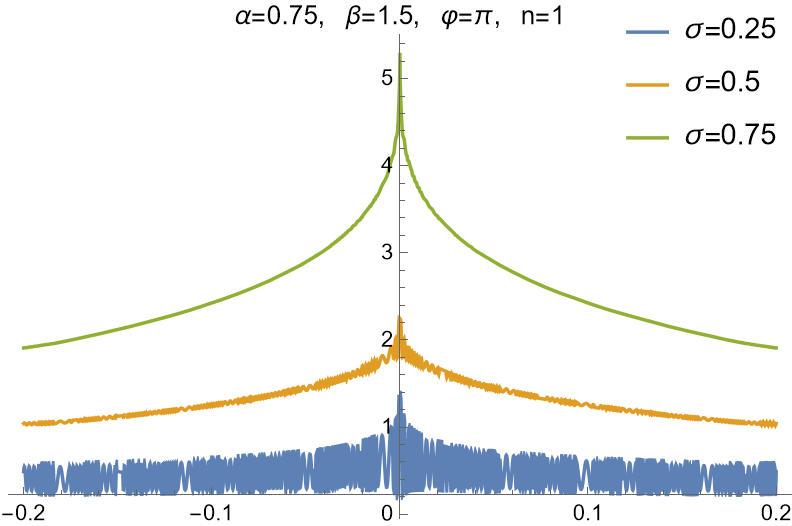}
\caption{$\alpha=0.75$, $\beta=1.5$}
\label{n133}
\end{subfigure}%
\vspace{1 cm}
\captionsetup{labelformat=empty}
\caption{Figures (1-i) through (1-ix) show $|\xi|^{n-\sigma}|\mathcal{F}
E_{\alpha,\beta}(- |\cdot|^{\sigma})(\xi)|$
near the origin in the one-dimensional case.}
\end{figure}
\unskip
\begin{figure}[H]
\centering
\begin{subfigure}{.48\textwidth}
  \centering
  \includegraphics[width=1\linewidth]{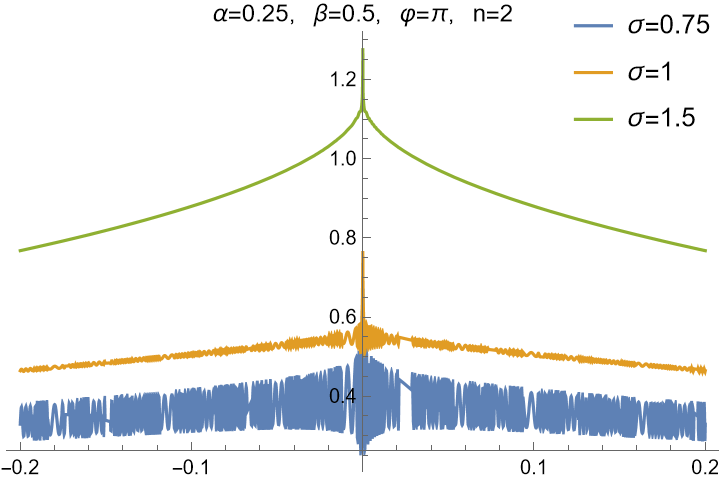}
\caption{$\alpha=0.25$, $\beta=0.5$}
\label{n211}
\end{subfigure}%
\hfill
\begin{subfigure}{.48\textwidth}
  \centering
  \includegraphics[width=1\linewidth]{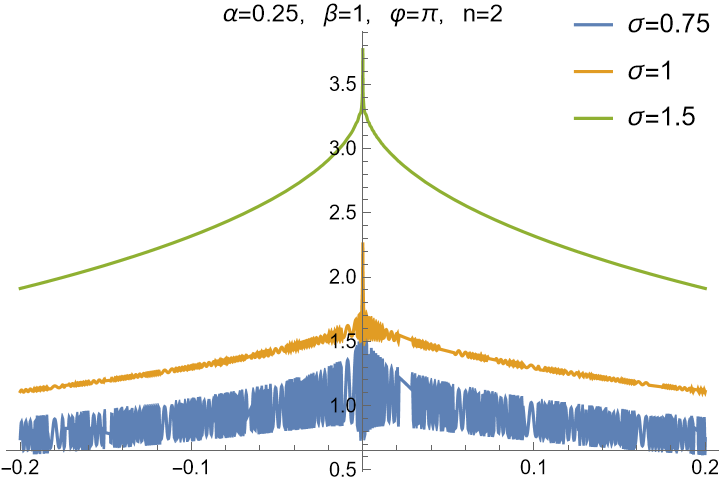}
\caption{$\alpha=0.25$, $\beta=1$}
\label{n212}
\end{subfigure}
\end{figure}
\unskip
\begin{figure}[H]
\ContinuedFloat
\centering
\begin{subfigure}{.48\textwidth}
\centering
  \includegraphics[width=1\linewidth]{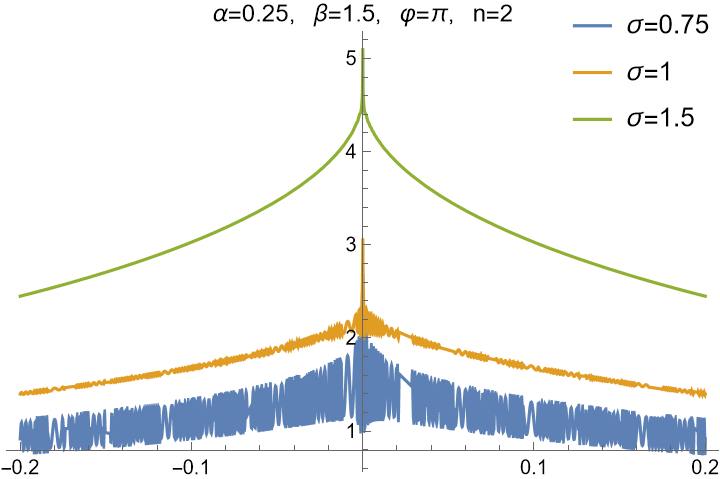}
\caption{$\alpha=0.25$, $\beta=1.5$}
\label{n213}
\end{subfigure}%
\hfill
\begin{subfigure}{.48\textwidth}
\centering
  \includegraphics[width=1\linewidth]{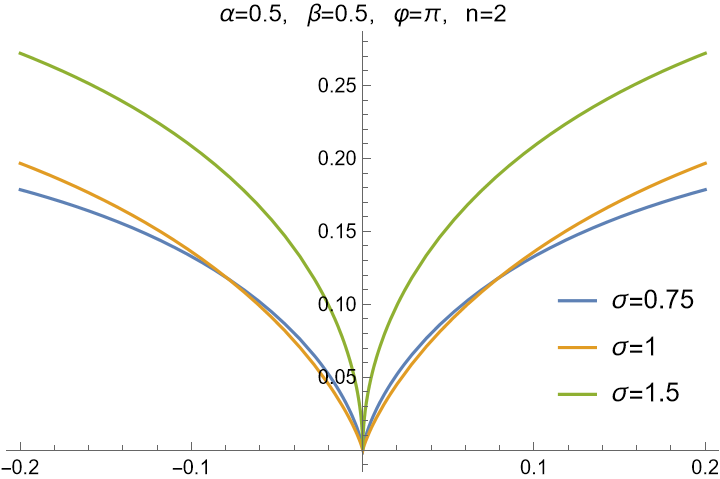}
\caption{$\alpha=0.5$, $\beta=0.5$}
\label{n221}
\end{subfigure}
\end{figure}
\unskip
\begin{figure}[H]
\ContinuedFloat
\centering
\begin{subfigure}{.48\textwidth}
\centering
  \includegraphics[width=1\linewidth]{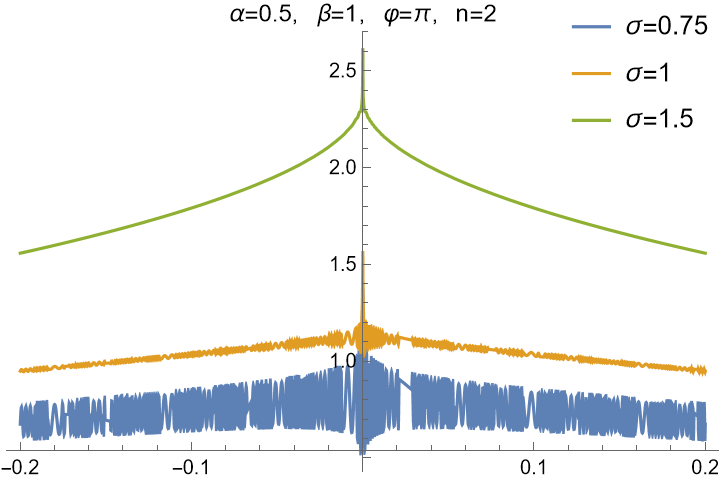}
\caption{$\alpha=0.5$, $\beta=1$}
\label{n222}
\end{subfigure}%
\hfill
\begin{subfigure}{.48\textwidth}
\centering
  \includegraphics[width=1\linewidth]{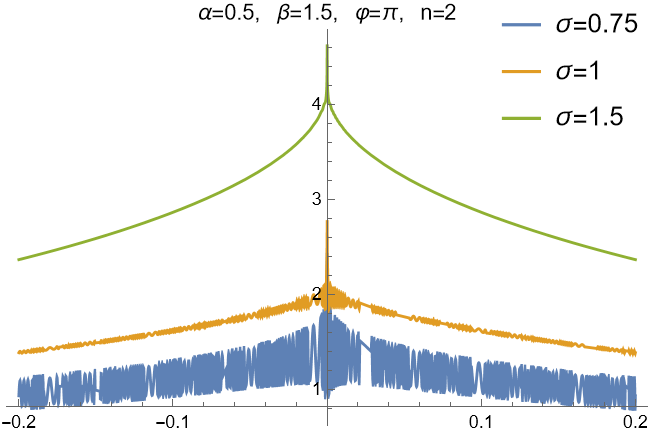}
\caption{$\alpha=0.5$, $\beta=1.5$}
\label{n223}
\end{subfigure}
\end{figure}
\unskip
\begin{figure}[H]
\ContinuedFloat
\centering
\begin{subfigure}{.48\textwidth}
\centering
  \includegraphics[width=1\linewidth]{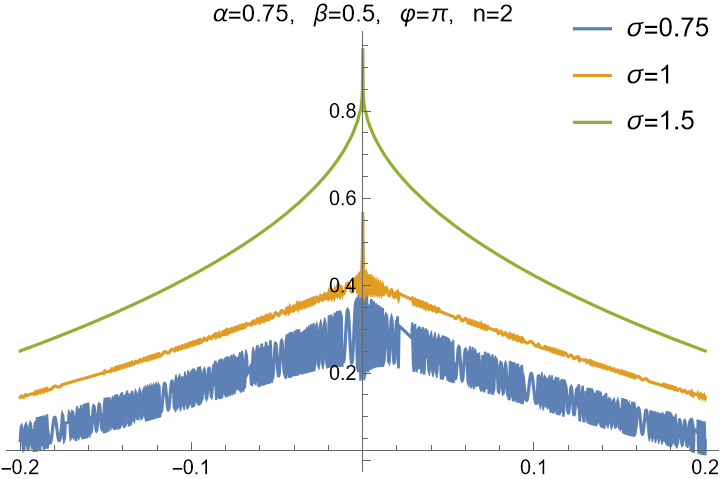}
\caption{$\alpha=0.75$, $\beta=0.5$}
\label{n231}
\end{subfigure}%
\hfill
\begin{subfigure}{.48\textwidth}
  \centering
  \includegraphics[width=1\linewidth]{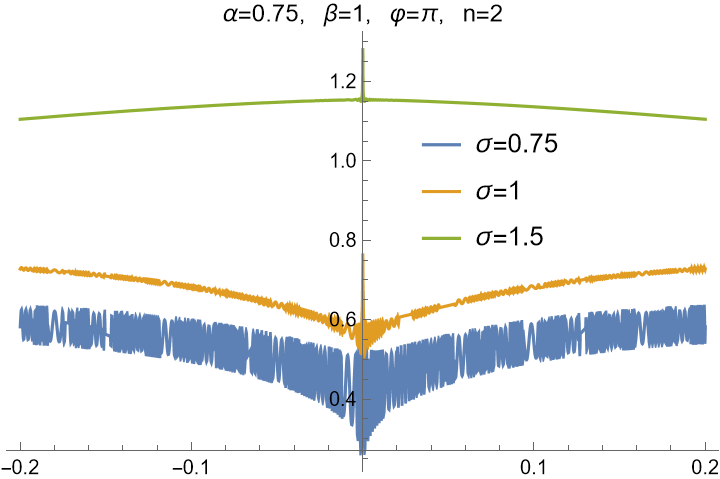}
\caption{$\alpha=0.75$, $\beta=1$}
\label{n232}
\end{subfigure}
\end{figure}
\unskip
\begin{figure}[H]
\ContinuedFloat
\centering
\begin{subfigure}{.48\textwidth}
\centering
  \includegraphics[width=1\linewidth]{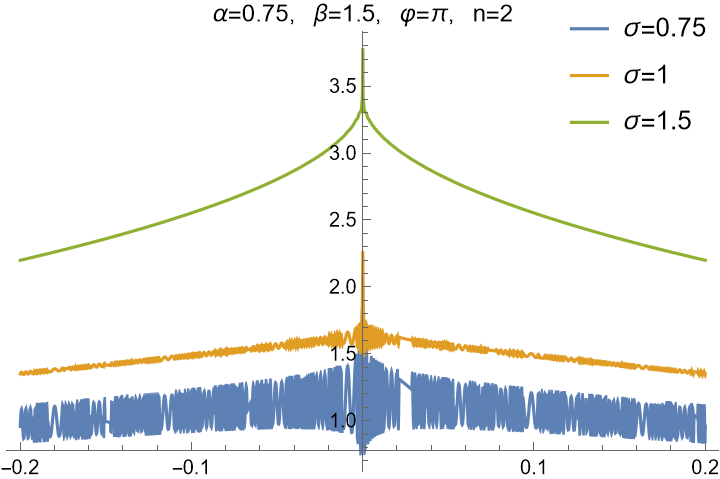}
\caption{$\alpha=0.75$, $\beta=1.5$}
\label{n233}
\end{subfigure}%
\vspace{1 cm}
\captionsetup{labelformat=empty}
\caption{Figures (2-i) through (2-ix) show $|\xi|^{n-\sigma}|\mathcal{F}
E_{\alpha,\beta}(- |\cdot|^{\sigma})(\xi)|$
near the origin in the two-dimensional case.}
\end{figure}
\unskip
\begin{figure}[H]
\centering
\begin{subfigure}{.48\textwidth}
  \centering
  \includegraphics[width=1\linewidth]{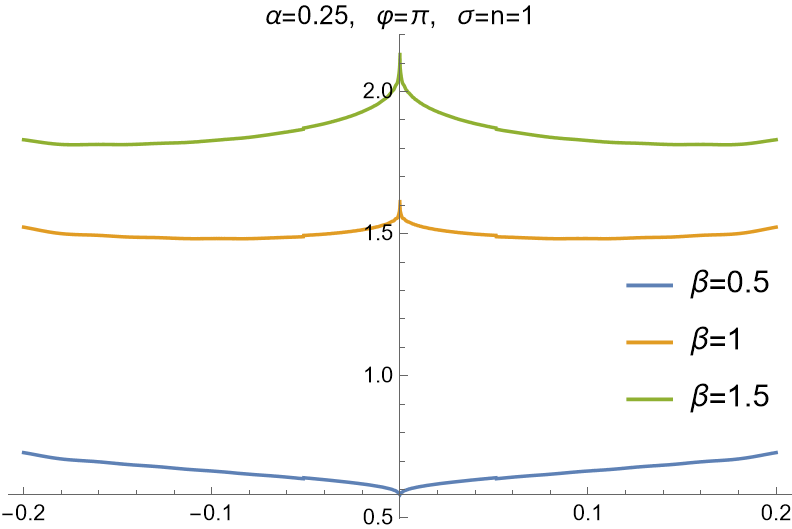}
\caption{$\alpha=0.25$, $\sigma=n=1$}
\label{n1e1}
\end{subfigure}%
\hfill
\begin{subfigure}{.48\textwidth}
  \centering
  \includegraphics[width=1\linewidth]{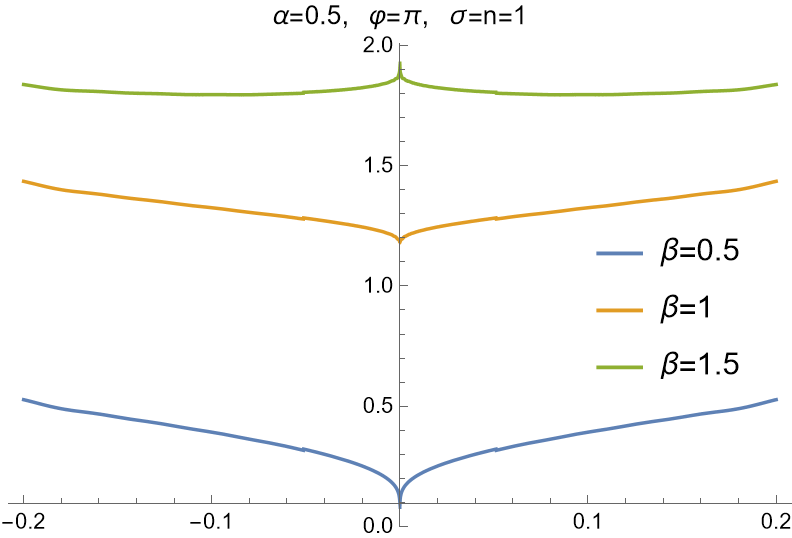}
\caption{$\alpha=0.5$, $\sigma=n=1$}
\label{n1e2}
\end{subfigure}
\vspace{- 2cm}
\end{figure}
\begin{figure}[H]
\ContinuedFloat
\centering
\begin{subfigure}{.48\textwidth}
\centering
  \includegraphics[width=1\linewidth]{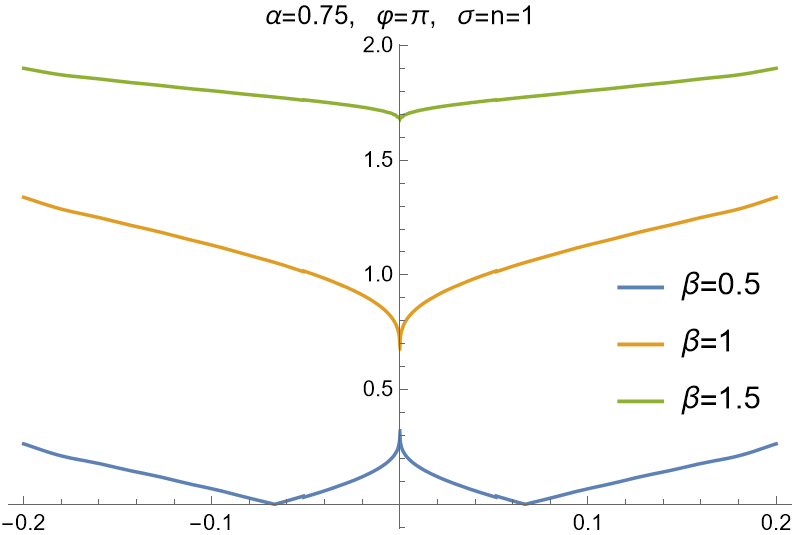}
\caption{$\alpha=0.75$, $\sigma=n=1$}
\label{n1e3}
\end{subfigure}%
\vspace{1 cm}
\captionsetup{labelformat=empty}
\caption{Figures (3-i) through (3-iii) show $|\mathcal{F}
E_{\alpha,\beta}(- |\cdot|^{\sigma})(\xi)|/|\log{|\xi|}|$
near the origin in the one-dimensional case.}
\end{figure}
\begin{figure}[H]
\centering
\begin{subfigure}{.48\textwidth}
\centering
  \includegraphics[width=1\linewidth]{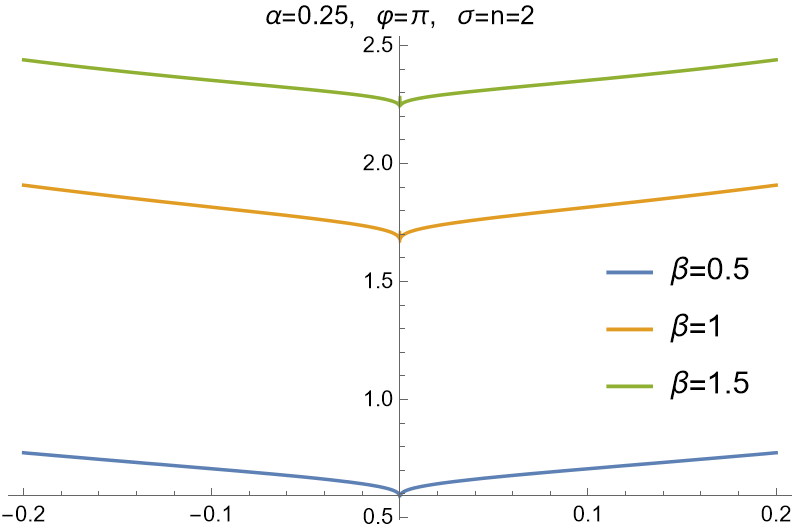}
\caption{$\alpha=0.25$, $\sigma=n=2$}
\label{n2e1}
\end{subfigure}%
\hfill
\begin{subfigure}{.48\textwidth}
\centering
  \includegraphics[width=1\linewidth]{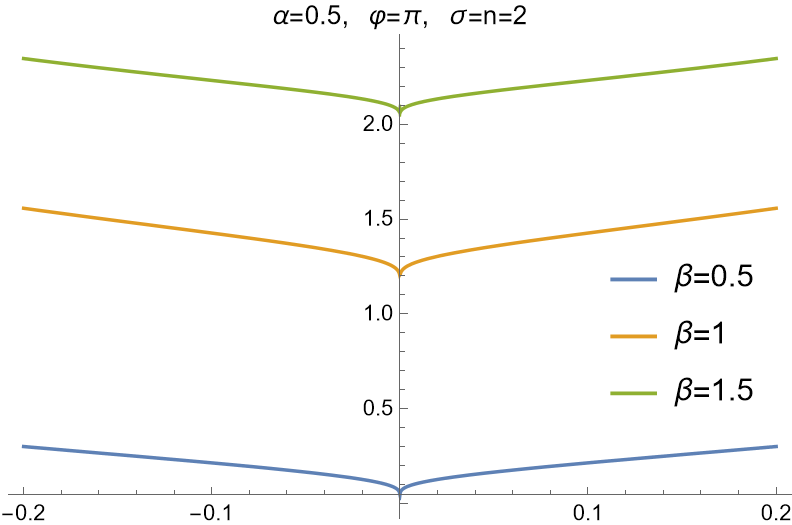}
\caption{$\alpha=0.5$, $\sigma=n=2$}
\label{n2e2}
\end{subfigure}
\end{figure}
\vspace{-1 cm}
\begin{figure}[H]
\ContinuedFloat
\centering
\begin{subfigure}{.48\textwidth}
\centering
  \includegraphics[width=1\linewidth]{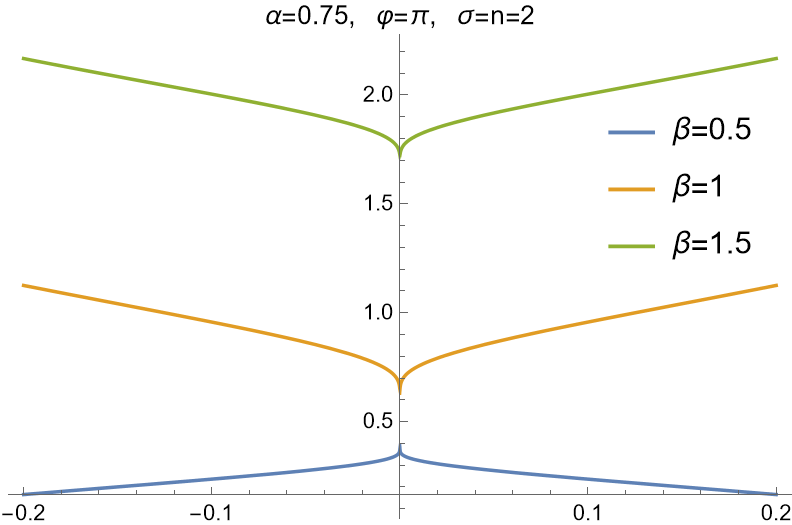}
\caption{$\alpha=0.75$, $\sigma=n=2$}
\label{n2e3}
\end{subfigure}
\vspace{1 cm}
\captionsetup{labelformat=empty}
\caption{Figures (4-i) through (4-iii) show $|\mathcal{F}
E_{\alpha,\beta}(- |\cdot|^{\sigma})(\xi)|/|\log{|\xi|}|$
near the origin in the two-dimensional case.}
\end{figure}



\bibliographystyle{spmpsci}
\bibliography{mtfrq}

\bigskip  

%

\end{document}